\documentclass{aims}
\usepackage{amsmath}
  \usepackage{paralist}    
  \usepackage{graphics} 
  \usepackage{epsfig} 
\usepackage{graphicx}  \usepackage{epstopdf}
 \usepackage[colorlinks=true]{hyperref}
\hypersetup{urlcolor=blue, citecolor=red}

\usepackage{tikz}
\usetikzlibrary{math,plotmarks}
\usepackage{subcaption}

\def\currentvolume{X}
 \def\currentissue{X}
  \def\currentyear{200X}
   \def\currentmonth{XX}
    \def\ppages{X--XX}
     \def\DOI{10.3934/xx.xx.xx.xx}

\newtheorem{theorem}{Theorem}[section]

\newtheorem{lemma}[theorem]{Lemma}
\newtheorem{prop}{Proposition}

\theoremstyle{definition}
\newtheorem{definition}[theorem]{Definition}
\newtheorem{remark}{Remark}
\newtheorem{example}{Example}

\title[Symplectic integration using Clebsch variables] 
      {Symplectic integration of PDEs \\ using Clebsch variables}

\author[Robert I McLachlan and Christian Offen and Benjamin K Tapley]{}

\subjclass{Primary: 37M15, 65P10; Secondary: 37K05, 35Q31, 37J15, 53D20}
 \keywords{Symplectic integration, Lie-Poisson system, Burgers' equation, Euler's equation}

 \email{r.mclachlan@massey.ac.nz}
 \email{c.offen@massey.ac.nz}
 \email{benjamin.tapley@ntnu.no}

\thanks{This research was supported by the Marsden Fund of the Royal Society Te Ap\={a}rangi.}

\thanks{$^*$ corresponding author. \href{mailto:c.offen@massey.ac.nz}{c.offen@massey.ac.nz}}

\def\d{\mathrm{d}}
\def\p{\partial }

\def\D{\mathrm{D}}
\def\id{\mathrm{id}}

\def\e{\epsilon}

\def\Z{\mathbb{Z}}

\def\R{\mathbb{R}}

\DeclareRobustCommand\marksymbol[2]{\tikz[#2,scale=1.2]\pgfuseplotmark{#1};}
\DeclareRobustCommand{\CollLine}{\raisebox{2pt}{\tikz{\draw[-.,red,dash pattern={on 7pt off 2pt on 3pt off 2pt},line width = 1.5pt](0,0) -- (9mm,0);}}}
\DeclareRobustCommand{\RefLine}{\raisebox{2pt}{\tikz{\draw[-,black,dash pattern={on 7pt off 2pt},line width = 1.5pt](0,0) -- (9mm,0);}}}
\DeclareRobustCommand{\ConvLine}{\raisebox{2pt}{\tikz{\draw[-,blue,line width = 1.5pt](0,0) -- (9mm,0);}}}

\begin{document}
\maketitle

\centerline{\scshape Robert I McLachlan and Christian Offen$^*$}
\medskip
{\footnotesize
 \centerline{Institute of Fundamental Sciences }
   \centerline{Massey University}
   \centerline{Private Bag 11 222, Palmerston North, 4442, New Zealand}
} 

\medskip

\centerline{\scshape Benjamin K Tapley}
\medskip
{\footnotesize
 \centerline{ Department of Mathematical Sciences}
 \centerline{Norwegian University of Science and Technology}
   \centerline{Sentralbygg 2, Gl{\o}shaugen, Norway}
}

\bigskip

 \centerline{(Communicated by the associate editor name)}

\begin{abstract}

Many PDEs (Burgers' equation, KdV, Camassa-Holm, Euler's fluid equations,\textellipsis) can be formulated as infinite-dimensional Lie-Poisson systems. These are Hamiltonian systems on manifolds equipped with Poisson brackets.
The Poisson structure is connected to conservation properties and other geometric features of solutions to the PDE and, therefore, of great interest for numerical integration.
For the example of Burgers' equations and related PDEs we use Clebsch variables to lift the original system to a collective Hamiltonian system on a symplectic manifold whose structure is related to the original Lie-Poisson structure. On the collective Hamiltonian system a symplectic integrator can be applied.
Our numerical examples show excellent conservation properties and indicate that the disadvantage of an increased phase-space dimension can be outweighed by the advantage of symplectic integration.

\end{abstract}

\section{Motivation}




Partial differential equations (PDEs) often exhibit interesting structure preserving properties, for example conserved quantities. In many examples, a deeper understanding of the structures can be achieved by viewing the PDE as the Lie-Poisson equation associated to an infinite-dimensional Lie group.
This means solutions to the PDE correspond to motions of a Hamiltonian system defined on the dual of the Lie-algebra of a Fr\'echet Lie-group.
Examples include Euler's equations for incompressible fluids, Burgers' equation, equations in magnetohydrodynamics, the Korteweg-de Vries equation, the superconductivity equation, charged ideal fluid equations, the Camassa-Holm equation and the Hunter-Saxton equation \cite{Vizman08}.
Conserved quantities turn out to be related to the fact that the Hamiltonian flow preserves the Lie-Poisson bracket. 
%
This makes Lie-Poisson structures interesting for structure preserving integration. We will give a brief review of Hamiltonian systems on Poisson manifolds in section \ref{sec:Intro}.

An approach to construct Lie-Poisson integrators, which works universally in the finite-dimensional setting, is to translate the Lie-Poisson system on a Lie-group $G$ to a Hamiltonian system on the tangent bundle $TG$ with a $G$-invariant Lagrangian. Using a variational integrator one obtains a Poisson-integrator for the original system \cite{DescreteEulerLiePoisson}. These integrators, however, can be extremely complicated \cite[p. 1526]{CollectiveIntegrators}. Moreover, the fact that exponential maps do not constitute local diffeomorphisms for infinite-dimensional manifolds restricts the approach to a finite-dimensional setting. 
Other approaches for energy preserving integration of finite dimensional Poisson systems with good preservation properties, e.g.\ preservation of linear symmetries or (quadratic) Casimirs, include \cite{ BRUGNANO20123890,Cohen2011,DiscreteGradientsPresIntegrals}. For a recent review article on Lie-Poisson integrators we refer to \cite{LiePoissonReview}.

Let us return to the infinite-dimensional setting. For numerical computations a PDE needs to be discretised in space. In the Lie-Poisson setting this corresponds to an approximation of the dual of a Lie algebra $\mathfrak g^\ast$ by a finite-dimensional space. The space $\mathfrak g^\ast$ typically corresponds to some space of $\R^k$-valued functions defined on a manifold. The most natural way of discretising $\mathfrak g^\ast$ is to introduce a grid on the manifold and identify a function with the values it takes over the grid. In this way we naturally obtain a finite-dimensional approximation of $\mathfrak g^\ast$. However, the approximation does not inherit a Poisson structure in a natural way, as we will see in the example of the Burgers' equation (remark \ref{rem:noheritage}). Therefore, finding a spatial discretisation with good structure preserving properties is a challenge.

Lie-Poisson systems $(\mathfrak g^\ast,\{,\},H)$ can be realised as collective Hamiltonian systems $(M,\Omega,H\circ J)$ on symplectic manifolds, where $J \colon M \to \mathfrak g^\ast$ is a Poisson map. The flow of $(M,\Omega,H\circ J)$ maps fibres of $J$ to fibres of $J$ and is symplectic. Therefore, it decends to a Poisson map on the original system $(\mathfrak g^\ast,\{,\},H)$.
Since the Hamiltonian vector field to $H\circ J$ on $(M,\Omega)$ is $J$-related to the Hamiltonian vector field to $H$ on $(\mathfrak g^\ast,\{,\})$, motions of $(M,\Omega,H\circ J)$ decend to motions of $(\mathfrak g^\ast,\{,\},H)$. 

The reason to consider a collective system for numerical integrations rather than the Lie-Poisson system directly is that the symplectic structure can easily be preserved under spacial discretisations and widely applicable, efficient symplectic integrators are available \cite{GeomIntegration}. The challenge of integrating $(\mathfrak g^\ast,\{,\},H)$ in a structure preserving way thus shifts to finding a realisation, i.e.\ $(M,\omega)$ and $J \colon M \to \mathfrak g^\ast$, such that all initial conditions of interest lie in the image of $J$ and such that the system $(M,\Omega,H\circ J)$ is practical to work with.

A practical choice for a realisation is where $J$ is a {\em Clebsch map} \cite{MarsdenCoadjointOrbits}:
let $X$ be a Riemannian manifold and let $M = T^\ast \mathcal C^\infty(X,\R^k) \cong C^\infty(X,\R^k) \times C^\infty(X,\R^k)^\ast$, where $ C^\infty(X,\R^k)^\ast$ is identified with $C^\infty(X,\R^k)$ via the $L^2$ pairing. The vector space $M$ is equipped with the symplectic form
\[
\Omega((u_1,u_2),(v_1,v_2))
= \int_X (\langle u_1, v_2\rangle_{\R^k}  - \langle u_2, v_1\rangle_{\R^k} )\, \d \mathrm{vol}_X,
\]
where $\langle .,.\rangle_{\R^k}$ denotes the scalar product in $\R^k$.
For an element $(f,g)\in M$ we denote the post-composition of $f$ and $g$ by the projection map to the $j^{\mathrm{th}}$ component of $\R^k$ by $q^j(f)$ and $p_j(g)$, respectively. In other words, $q^1,\ldots,q^k,p_1,\ldots,p_k$ are maps $M \to \mathcal C^\infty(X,\R)$ such that for $x \in X$
\[
(f(x),g(x)) = \Big(\big(q^1(f)(x),\ldots,q^k(f)(x)\big),\big(p_1(g)(x),\ldots,p_k(g)(x)\big)\Big).
\]
Identifying tangent spaces of the vector space $M$ with itself, we may write $\Omega$ as
\[
\Omega 
= \int_X \left(\sum_{j=1}^k \d q^j \wedge \d p_j \right) \d \mathrm{vol}_X
= \int_X \langle \d {q} \wedge \d {p} \rangle_{\R^k} \, \d \mathrm{vol}_X,
\]
where $q = (q^1,\ldots,q^k)$ and $p = (p_1,\ldots,p_k)$\footnote{The notation is natural when considering $M$ as a Fr\'echet manifold over $C^\infty(X,\R)$ or $C^\infty(X,\R^k)$ with coordinates $(q^1,\ldots,q^k,p_1,\ldots,p_k)$ or $(q,p)$, respectively.}.
If $J\colon M \to \mathfrak g^\ast$ is a realisation of a Lie Poisson system $(\mathfrak g^\ast,\{,\})$, then $J$ is called a {\em Clebsch map} and $(q,p)$ are called {\em Clebsch variables}. In Clebsch variables Hamilton's equations for $\bar H =H\circ J \colon M \to \R$ are in canonical form, i.e.\
\[
q_t = \frac{\delta \bar H}{\delta p}, \qquad p_t = -\frac{\delta \bar H}{\delta q},
\]
where $\frac{\delta \bar H}{\delta q}$ and $\frac{\delta \bar H}{\delta p}$ are variational derivatives.
The reason why Clebsch variables are a natural choice of coordinates for a structure preserving setting is that if $X$ is discretised using a mesh then the integral in the expression for $\Omega$ naturally becomes a (weighted) sum over all mesh points and Hamilton's equations for the discretisation of the collective system $(M,\Omega,H\circ J)$ are in (a scaled version of the) canonical form. This means the system can be integrated using a symplectic integrator like, for instance, the midpoint rule.
The setting is summarised in table \ref{table:overview}.

\begin{table}
\centering
\def\arraystretch{1.5}
\begin{tabular}{ p{6cm} | p{6cm} } 
Continuous system & Spatially discretised system \\
 \hline\hline 
Collective Hamiltonian system on an infinite-dimensional symplectic vector space in Clebsch variables
\begin{center}
$q_t = \frac{\delta \bar H}{\delta p}, \quad p_t = -\frac{\delta  \bar H}{\delta q}.$
\end{center}
Exact solutions preserve the symplectic structure, the Hamiltonian $\bar H=H\circ J$, all quantities related to the Casimirs of the original PDE and the fibres of the Clebsch map $J(q,p)=u$.
& Canonical Hamiltonian ODEs in $2N$ variables
\begin{center}
$\hat q_t = \nabla_{\hat p} \hat {\bar H}, \quad \hat p_t = - \nabla_{\hat q} \hat {\bar H}.$
\end{center}
The exact flow preserves the symplectic structure and the Hamiltonian $\hat {\bar H}$.

Time-integration with the midpoint rule is symplectic.
\\
\hline
Original PDE, interpreted as a Lie-Poisson equation
\begin{center}
$u_t = \mathrm{ad}^\ast_{\frac {\delta H}{\delta u}}u.$
\end{center}
Exact solutions preserve the Poisson structure, the Hamiltonian $H$ and all Casimirs.
 & Non-Hamiltonian ODEs in $N$ variables 
 \begin{center}
$\hat u_t = K(\hat u) \nabla_{\hat u} \hat H, \qquad K^T=-K.$
\end{center}
Exact solutions conserve $\hat H$.

Time-integration with the midpoint rule is {\em not} symplectic.
\\
 \hline
\end{tabular}
\caption{Overview of the setting.}
\label{table:overview}
\end{table}

The symplectic system in Clebsch variables has, after spatial discretisation, twice as many variables as the discretisation of the PDE in the original variables.
An increase in the amount of variables needs some justification because it does not only lead to more work per integration step but, thinking of multi-step methods versus one-step methods, can also lead to worse stability behaviour \cite[XV]{GeomIntegration}. Moreover, integrating a lifted, symplectic system with a symplectic integrator instead of the original system with a non-symplectic integrator is not necessarily of any advantage. If, for instance, we integrate the Hamiltonian system
\begin{align*}
\dot u &= F(u) \qquad \;\;    \,= \phantom{-}\nabla_p\langle F(u),p\rangle \\
\dot p &= - \D F(u)^T p   = -\nabla_u\langle F(u),p\rangle
\end{align*}
rather than the system $\dot u = F(u)$ directly then preserving the symplectic structure in a numerical computation does not have any effect:
in this example the symplectic structure is artificially introduced and not related to the original system. This illustrates that using symplectic integrators is not an end in itself.
It is the presence of a Poisson structure and its interplay with the symplecticity of the collective system which can justify doubling the amount of variables as our numerical examples will indicate.

Let us provide examples for the application of Clebsch variables.
Euler's equation in hydrodynamics for an ideal incompressible fluid with velocity $u$ and pressure $\rho$ on a 3-dimensional compact, Riemannian manifold $X$ with boundary $\partial X$ or a region $X\subset\R^3$ are given as
\[
u_t + u \cdot \nabla u = -\nabla \rho, \qquad \mathrm{div}\, u = 0, \qquad u|_{\partial X} \text{ is parallel to $\partial X$.}
\]
Elements in the dual of the Lie-algebra $\chi^\ast_{\mathrm{vol}}$ to the Fr\'echet Lie-group of volume preserving diffeomorphisms $\mathcal{D}_{\mathrm{vol}}$ can be considered as 2-forms on $X$. Using $\nabla \times u \,\widehat{=}\, \d u^{\flat}$ Euler's equations correspond to motions on the Lie-Poisson system to $\mathcal{D}_{\mathrm{vol}}$ with Hamiltonian $H(\sigma) = \frac 12 \int_X \langle \Delta^{-1} \sigma,\sigma \rangle \d \mathrm{vol}_X$, where $\Delta$ is the Laplace-DeRham operator and $\langle,\rangle$ the metric pairing of 2-forms \cite{MarsdenCoadjointOrbits}.

A Clebsch map $J\colon M\to \chi^\ast_{\mathrm{vol}}$ can be obtained as the momentum map of the cotangent lifted action of the action $(\eta, f) \mapsto f \circ \eta^{-1}$ of $\mathcal{D}_{\mathrm{vol}}$ on $\mathcal{C}^\infty(X,\R)$.
However, $J$ is not surjective and flows with non-zero hydrodynamical helicity cannot be modelled. To overcome this issue one can consider $M = \mathcal C^\infty(X,S^2)$, where $S^2$ is the 2-sphere. The symplectic form $\sigma_{S^2}$ on the sphere induces the symplectic form $\Omega = \int_X \sigma_{S^2} \d \mathrm{vol}_X$ on $M$. We can define $J \colon M \to \chi^\ast_{\mathrm{vol}}$ as $J(s) = s^\ast \sigma_{S^2}$, where $s^\ast \sigma_{S^2}$ denotes the pull-back of $\sigma_{S^2}$ to a 2-form on $X$ which can be interpreted as an element in $\chi^\ast_{\mathrm{vol}}$. The map $J$ is called a {\em spherical Clebsch map} and initial conditions with non-zero helicity are admissible. However, the helicity remains quantised \cite{TopoMeaningSphericalClebsch}.
Spherical Clebsch maps have been used for computational purposes in \cite{chern2016schrodinger}: after a discretisation of the domain $X$, solutions to the (regularised) hydrodynamical equations are approximated by integrating the corresponding set of ODEs on the product $\Pi_{\mathrm{mesh}(X)} S^2$ while preserving the spheres using a projection method (not preserving the symplectic form $\sum_{\mathrm{mesh}(X)}\sigma_{S^2}$, though). 

In the case of Hamiltonian ODEs on (finite-dimensional) Poisson spaces $(\mathfrak g^\ast,\{,\})$, no spatial discretisation is necessary. This setting applies to the rigid-body equations, for instance \cite{Marsden78}. In the ODE setting, the authors of \cite{CollectiveIntegrators} apply symplectic integrators to the collective systems $(M,\Omega,H\circ J)$ with the property that the discrete flow preserves the fibres of $J$. Such integrators are called {\em collective integrators}. Their flow descends to a Poisson map on the original system $(\mathfrak g^\ast,\{,\},H)$ such that one obtains a Poisson integrator for $(\mathfrak g^\ast,\{,\},H)$.

In this paper, we show how the collective integrator idea can be used in the infinite-dimensional setting, i.e.\ for Lie-Poisson systems to infinite-dimensional Lie-groups.
In particular, we will consider the inviscid Burgers' equation
\[
u_t + u u_x = 0
\]
with $u(t,.) \in \mathcal C^\infty(S^1,\R)$. The $L^2$-norm of $u(t,.)$ as well as the quantity \[\int_{S^1} \sqrt{|u(t,.)|}\d x\] are conserved quantities. They constitute the Hamiltonian and Casimirs of the Lie-Poisson formulation of the problem. Setting $u = q_x p$ we obtain the following set of PDEs 
\[
q_t = -\frac 13 q^2_xp,
\qquad
p_t = -\frac 13 (q_xp^2)_x
\]
with $q(t,\cdot) \in \mathcal C^\infty(S^1,S^1)$ and $p(t,\cdot) \in \mathcal C^\infty(S^1,\R)$ which is the collective system.
The variables $q,p$ may be regarded as Clebsch variables (right in the middle between classical and spherical Clebsch variables).

We will also experiment with the following more complicated PDE which fits into the same setting as the inviscid Burgers' equation.
\begin{equation*}
u_t = 3uu_x - \frac{9}{4}u^2u_x - u_xu_{xx} - 3u_x^2u_{xx} - 2uu_{xxx} - 2uu_xu_{xxx} - 6uu_{xx}^2
\end{equation*}
It has the conserved quantity $H(u)=\int_{S^1} (u^2+u_x^2-1/2 u^3+u_x^3) \d x$ as well as $\int_{S^1} \sqrt{|u|}\d x$ in time. In Clebsch variables we have 
\begin{align*}
q_t &= \frac{\delta \bar H}{\delta p} \;\;\, = q_x\Big(q_xp-\frac 34 (q_xp)^2-((q_xp)_x+\frac 32 (q_xp)_x^2)_x\Big)\\
p_t &= -\frac{\delta \bar H}{\delta p} = p \Big( \frac 32 (q_xp)^2-q_xp+\big((q_xp)_x+\frac 32 (q_xp)_x^2\big)_x \Big)_x.
\end{align*}

The PDEs are discretised in space by introducing a periodic grid on $S^1$ and replacing the integral in $H$ by a sum. In this way we obtain a system of Hamiltonian ODEs in canonical form.

Integration using the symplectic midpoint rule yields an integrator with excellent structure preserving properties like bounded energy and Casimir errors, although it does not preserve the fibres of $J$ and therefore does not descend to a Poisson integrator. 
The good behaviour is linked to the symplecticity of the collective system which is preserved exactly by the midpoint rule. Therefore, the conservation properties survive even when the equation is perturbed within the class of Hamiltonian PDEs. This robustness can be an advantage over more traditional ways of discretising the PDE directly since these make use of structurally simple symmetries of the equation that are immediately destroyed when higher order terms are introduced. Our numerical experiments indicate that the advantage of symplectic integration can outweigh the disadvantage of doubling the variables from $u$ to $(q,p)$.

\section{Introduction}\label{sec:Intro}

Let us briefly review the setting of Hamiltonian systems on Poisson manifolds. For details we refer to \cite{Marsden99}.

\begin{definition}[Poisson manifold and Poisson bracket]
A {\em Poisson manifold} $P$ is a smooth manifold together with an $\R$-bilinear map 
\[\{\cdot,\cdot\}  \colon \mathcal C^\infty(P) \times  \mathcal C^\infty(P) \to  \mathcal C^\infty(P)\]
satisfying
\begin{itemize}
\item
$\{f,g\}=-\{g,f\}$ (skew-symmetry),
\item
$\{f,\{g,h\}\}+\{g,\{h,f\}\}+\{h,\{f,g\}\}=0$ (Jacobi identity),
\item
$\{fg,h\}=f\{g,h\}+g\{f,h\}$ (Leibniz's rule).
\end{itemize}
The map $\{\cdot,\cdot\}$ is called the {\em Poisson bracket}.
\end{definition}

\begin{example}\label{ex:LiePoissonStr}
If $G$ is a (Fr\'echet-) Lie-group with Lie-algebra $\mathfrak g$ and dual $\mathfrak g^\ast$ then
\begin{equation}\label{eq:LiePoissonBracket}
\{f,g\}(w) = \left\langle w, \left[\frac{\delta f}{\delta w},\frac{\delta g}{\delta w}\right]\right\rangle, \qquad w\in \mathfrak g^\ast, \, f,g \in \mathcal C^\infty(\mathfrak g^\ast)
\end{equation}
is a (Lie-) Poisson bracket on $\mathfrak g^\ast$, where $\langle\cdot,\cdot\rangle$ denotes the duality pairing of $\mathfrak g^\ast$ and $\mathfrak g$, $[\cdot,\cdot]$ denotes the Lie bracket on $\mathfrak g$ and $\frac{\delta f}{\delta w} \in \mathfrak g$ is defined by
\[
\forall v \in \mathfrak g^\ast : \quad \D f|_{w} (v) = \left\langle v, \frac{\delta f}{\delta w}\right\rangle
\]
with Fr\'echet derivative $\D$.
\end{example}

\begin{definition}[Hamiltonian system and Hamiltonian motion]
A {\em Hamiltonian system} $(P, \{\cdot,\cdot\},H)$ is a Poisson manifold $(P, \{\cdot,\cdot\})$ together with a smooth map $H \colon P \to \R$. The {\em Hamiltonian vectorfield} $X_H$ to the system $(P, \{\cdot,\cdot\},H)$ is defined as the derivation $X_H = \{\cdot,H\}$. If $f \colon P \to \R$ is a smooth function, then the {\em motion of the system $(P, \{\cdot,\cdot\},H)$} in the coordinate $f$ is given by the differential equation $\dot f = \{f,H\}$, where the dot denotes a time-derivative. 
\end{definition}

\begin{example} A Hamiltonian system $(M, \omega,H)$ on a symplectic manifold $(M, \omega)$ constitutes a Hamiltonian system on the Poisson manifold $(M, \{\cdot,\cdot\})$. The Poisson bracket $\{\cdot,\cdot\}$ is defined by $\{f,g\} = \omega(X_f,X_g)$ where the vector fields $X_f$ and $X_g$ are defined by $\d f = \omega(X_f,\cdot)$ and $\d g = \omega(X_g,\cdot)$. If $M$ is $2n$-dimensional with local coordinates $q^1,\ldots,q^n,p_1,\ldots,p_n$ and $\omega = \sum_{j=1}^{n}\d q^j \wedge \d p_j$ then
\[X_H = \sum_{j=1}^n \frac{\p H}{\p p_j} \frac{\p }{\p q^j} - \frac{\p H}{\p q^j} \frac{\p }{\p p_j}.\]
The motions of the system are given by
\begin{align*}
\dot q^j &= \{q^j,H\} = X_H(q^j) = \; \;\; \frac{\p H}{\p p_j},\\
\dot p_j &= \{p_j,H\} = X_H(p_j) =  -\frac{\p H}{\p q^j}.
\end{align*}
with $j=1,\ldots,n$.
\end{example}

\begin{remark} For Hamiltonian systems on a finite-dimensional, symplectic manifold, there exist local coordinates such that the motions are given by 
\[
\dot z = S \nabla H(z),
\]
for a constant, skew-symmetric, non-degenerate matrix $S$. The analogue for finite-dimensional Poisson systems is that $S$ is allowed to be $z$ dependent and degenerate (but still skew-symmetric).
\end{remark}

\begin{remark}
Like in the symplectic case, the Hamiltonian is a conserved quantity under motions of the corresponding Hamiltonian system on a Poisson manifold.
Additionally, the Poisson structure encodes interesting geometric features of Hamiltonian motions. Casimir functions, which are real valued functions $f$ with $\{f,\cdot\}=0$ are conserved quantities (with no dependence on the Hamiltonian). While the only Casimirs are constants if the Poisson structure is induced by a symplectic structure, non-trivial Casimir functions are admissible in the Poisson case. Moreover, in a Poisson system a motion never leaves the coadjoint orbit in which it was initialised. We refer to \cite[Ch.10]{Marsden99} for proofs and more properties of Poisson manifolds. 
\end{remark}


In what follows we will present an integrator for Hamiltonian systems on the dual of the Lie-algebra of the group of diffeomorphisms on the circle. The setting covers, for example, Burgers' equation and perturbations. This shows how to apply the ideas of \cite{CollectiveIntegrators} in the infinite-dimensional setting of Hamiltonian PDEs.

\section{Lie-Poisson structure on $\mathrm{diff}(S^1)^\ast$}

Consider the Fr\'echet Lie-group $G=\mathrm{Diff}(S^1)$ of orientation preserving diffeomorphisms on the circle $S^1$. In the following we view $S^1$ as the quotient $\mathbb R / L \mathbb Z$ for $L>0$ with coordinate $x$ obtained from the universal covering $\mathbb R \to \mathbb R / L$. The Lie-algebra $\mathfrak g$ can be identified with the space of smooth vector fields on $\mathcal S^1$, where the Lie-bracket is given as the negative of the usual Lie-bracket of vector fields
\[
\left[ u \frac{\p}{\p x}, v \frac{\p}{\p x}\right] = (u_x v - v_x u) \frac{\p}{\p x}.
\]
Here, the prime denotes a derivative with respect to the coordinate $x$ on $S^1 = \mathbb R / L \mathbb Z$.\cite[Thm.43.1]{ConvenientSetting} The dual $\mathfrak g^\ast$ of the Lie algebra\footnote{which does not coincide with the functional analytic dual to $\mathfrak g$} can be identified with the quadratic differentials on the circle $\Omega^{\otimes 2}(S^1) = \{ u \cdot (\d x)^2 \, | \, u \in \mathcal C^\infty(S^1,\R)\}$. The dual pairing is given by 
\[
\left\langle u (\d x)^2, v \frac{\p}{\p x} \right \rangle
=
\int_{S^1} u(x)v(x) \d x.
\]
\cite[Prop. 2.5]{Khesin2009}
The coadjoint action of an element $\phi \in G$ on an element $ u (\d x)^2$ is given as
\[
\mathrm{Ad}^\ast_{\phi^{-1}}\left( u (\d x)^2\right)
=(u \circ \phi) \cdot \phi'^2 \cdot (\d x)^2 
= \phi^\ast \left( u (\d x)^2\right).
\]
We see that the coadjoint action on $u (\d x)^2$ preserves the zeros of $u$. The map $u$ will have an even number of zeros. Consider two consecutive zeros $a,b \in S^1$.
The integral
\[
\int_a^b \sqrt{|u(x)|} \d x
\]
is constant on the coadjoint orbit through $u (\d x)^2$ since the action corresponds to a diffeomorphic change of the integration variable in the above expression. It follows that the map $\Phi \colon \mathfrak g^\ast \to \R$ with 
\[\Phi(u (\d x)^2) = \int_{S^1}\sqrt{|u(x)|} \d x\]
is a Casimir for the Poisson structure on $\mathfrak g^\ast$. \cite{Khesin2009}
For $H \in \mathcal C^\infty(\mathfrak g^\ast, \R)$ Hamilton's equations are given as
\[
\frac{\d}{\d t} 
u(t,x) (\d x)^2
= \mathrm{ad}^\ast_{\frac{\delta H}{\delta u(t,\cdot) (\d x)^2}} \left(u(t,x) (\d x)^2 \right)
\]
or, identifying $\mathfrak g$ and $\mathfrak g^\ast$ with $\mathcal C^\infty(S^1,\R)$,
\[
u_t = \mathrm{ad}^\ast_{\frac{\delta H}{\delta u}} u.
\]
Here $\frac{\delta H}{\delta u}$ denotes the functional or variational derivative of $H$ and $\mathrm{ad}_\eta^\ast \colon \mathfrak g^\ast \to \mathfrak g^\ast$ the dual map to $\mathrm{ad}_\eta \colon \mathfrak g \to \mathfrak g$ given by
\[
\mathrm{ad}_\eta(\mu) = [\eta,\mu].
\]
\cite[Prop. 10.7.1.]{Marsden99}
\begin{lemma}\label{lemma:HamEQ}
Hamilton's equations can be rewritten as
\begin{equation}\label{eq:Hamongstar}
u_t = \left(\frac{\p}{\p x}u + u \frac{\p}{\p x}\right) \frac{\delta H}{\delta u}.
\end{equation}
\end{lemma}

\begin{proof}
Let $v \in \mathfrak g$, $u \in \mathfrak g^\ast$ (both identified with $\mathcal C^\infty(S^1,\R)$). Denoting the dual pairing between $\mathfrak g$ and $\mathfrak g^\ast$ by $\langle, \rangle$, we obtain
\begin{align*}
\left \langle \mathrm{ad}^\ast_{\frac{\delta H}{\delta u}} u , v \right \rangle
&= \left \langle  u , \mathrm{ad}_{\frac{\delta H}{\delta u}} v \right \rangle
= \left \langle  u , \left[ \frac{\delta H}{\delta u} , v\right] \right \rangle
= \left \langle  u , \left(\frac{\delta H}{\delta u}\right)_x\cdot v - \left(\frac{\delta H}{\delta u}\right)\cdot v_x   \right \rangle\\
&= \left \langle  u \cdot \left(\frac{\delta H}{\delta u}\right)_x ,  v \right \rangle
- \left \langle u \cdot  \left(\frac{\delta H}{\delta u}\right) , v_x   \right \rangle\\
&= \left \langle  u \cdot \left(\frac{\delta H}{\delta u}\right)_x ,  v \right \rangle
+  \left \langle \left( u \cdot  \left(\frac{\delta H}{\delta u}\right)\right)_x , v   \right \rangle,
\end{align*}
whereas the last equation follows using integration by parts. 
\end{proof}

\begin{example}\label{ex:Hong}
On $\mathfrak g^\ast$ consider the Hamiltonian
\[H(u) = \int_{S^1} \mathcal H(u^{\mathrm{jet}}(x)) \d x\]
with $\mathcal H \colon \R^{K+1} \to \R$ and the $K$-jet of the map $u$
\begin{align*}
u^{\mathrm{jet}}(x)
&:= (u(x),u_x(x),u_{x^2}(x),\ldots,u_{x^K}(x))\\
&:= \left( u(x),\left.\frac{\p u}{\p x}\right|_x, \left.\frac{\p^2 u}{\p x^2}\right|_x,\ldots,\left.\frac{\p^K u}{\p x^K}\right|_x\right).
\end{align*}
By lemma \ref{lemma:HamEQ}, Hamilton's equations are given as
\[
u_t = \left(\frac{\p}{\p x}u + u \frac{\p}{\p x}\right) \sum_{j=0}^{K} (-1)^j\frac{\p}{\p x^j} \left(\frac{\p \mathcal H}{\p u_{x^j}}(u^{\mathrm{jet}}) \right).
\]
For $\mathcal H(u)=-\frac 16 u^2$ we obtain the inviscid Burgers' equation $u_t+uu_x = 0$. 
\end{example}

\begin{remark}\label{rem:noheritage}
Using formula \eqref{eq:LiePoissonBracket} from example \ref{ex:LiePoissonStr} identifying $\mathfrak g \cong \mathcal C^\infty(S^1,\R)$ and $\mathfrak g^\ast \cong \mathcal C^\infty(S^1,\R)$, the Lie-Poisson bracket is given by
\[
\{F,G\}(u)
= \int_{S^1} \left( \frac{\d}{\d x} \left(\frac{\delta F}{\delta u}\right) \frac{\delta G}{\delta u}
-  \frac{\delta F}{\delta u}\frac{\d}{\d x} \left(\frac{\delta G}{\delta u}\right)\right)f \;\d x,
\]
where $\frac{\delta F}{\delta u}$ denotes the functional or variational derivative of $F$ at $u$. Discretising $S^1 \cong \R/\Z $ using a (periodic) grid with $N$ grid-points, we naturally obtain $\R^N$ as a discrete analog of $\mathfrak g^\ast$. However, the above Poisson structure does not pass naturally to $\R^N$. 
\end{remark}
\section{The collective system}

Let us construct a realisation $J\colon M \to \mathfrak g^\ast$ where $M$ is a symplectic vector space.
Consider the left-action of $g \in G=\mathrm{Diff}(S^1)$ on $q \in Q = \mathcal C^\infty(S^1,S^1)$ defined by $g . q = q \circ g^{-1}$.

\begin{lemma} The vector field $\hat v$ generated by the infinitesimal action of an element $v \in \mathfrak g \cong \mathfrak{X}(S^1)$ on $Q$ is given by the Lie-derivative $-\mathcal L_v$. Interpreting $v$ as an element in $\mathcal C^\infty(S^1,\R)$, this becomes $\hat v_q = - v \cdot q' \in \mathcal C^\infty(S^1,\R) \cong T_qQ$.
\end{lemma}

\begin{proof}
Let $g \colon(-\e,\e)\to \mathrm{Diff}(S^1)$ be a smooth curve with $g_0 = \id$ and $\left.\frac{\d}{\d t}\right|_{t=0} g_t = v \in \mathfrak g \cong \mathcal C^\infty(S^1,\R)$. Let $x \in S^1$. Deriving $x = g_t(g_t^{-1}(x))$ w.r.t.\ $t$ at $t=0$ we obtain
\[
\left.\frac{\d}{\d t}\right|_{t=0} g_t^{-1}(x) = -v(x).
\]
Let $q \in Q$. We have
\[
\hat v _q(x) 
= \left.\frac{\d}{\d t}\right|_{t=0} (g_t . q)(x)
= \left.\frac{\d}{\d t}\right|_{t=0} \left(q\circ g_t^{-1}\right)(x)
= - v(x) q'(x).
\]
\end{proof}

Let $M$ denote the cotangent bundle over $Q$, which is viewed as $T^\ast Q \cong Q \times \mathcal C^\infty(S^1,\R)$. The pairing of $(q,p) \in M$ with an element $v \in T_q Q \cong \mathcal C^\infty(S^1,\R)$ is given by
\[
\langle (q,p), v \rangle = \int_{S^1} p(x)v(x) \d x.
\]
A symplectic structure on $M$ is given by
\[
\Omega( (v^q,v^p), (w^q,w^p) ) = \int_{S^1} (w^p v^q- v^p w^q ) \d x.
\]
For $(q,p) \in M$ and $\bar H\colon M \to \R$ the maps $\frac{\delta \bar H}{\delta q}$ and $\frac{\delta \bar H}{\delta p}$ can be defined by
\begin{align*}
\D \bar H|_{(q,p)}(w^q,0)  = \int_{S^1}  \frac{\delta \bar H}{\delta q} w^q \d x,
\quad
\D \bar H|_{(q,p)}(0,w^p)  = \int_{S^1}  \frac{\delta \bar H}{\delta p} w^p \d x,
\end{align*}
where $\D$ denotes the G\^ateaux derivative.\footnote{Each maps $\frac{\delta \bar H}{\delta q}$ and $\frac{\delta \bar H}{\delta p}$ can depend on both $q$ and $p$ although this is not incorporated in the notation.} Now
\[
\D \bar H|_{(q,p)}(w^q,w^p) 
= \Omega\left( \left(\frac{\delta \bar H}{\delta p},-\frac{\delta \bar H}{\delta q}\right), (w^q,w^p) \right) 
\]
and Hamilton's equations can be written in the familiar looking form
\begin{equation}\label{eq:HamCanonicalLooking}
q_t = \;\;\,\frac{\delta \bar H}{\delta p}, \qquad p_t = -\frac{\delta \bar H}{\delta q}.
\end{equation}

We consider the cotangent lifted action of the aforementioned action of $G$ on $Q$ to obtain a Hamiltonian group action of $G$ on $M$ given by
\[g.(q,p) = (q \circ g^{-1}, p \circ g^{-1} \cdot (g^{-1})_x).\]
Alternatively, interpreting the fibre component of elements in $T^\ast Q$ as 1-forms the action is given by $g.(q,p\d x) = \left(q \circ g^{-1}, (g^{-1})^\ast (p\d x)\right)$.


\begin{prop}
The momentum map $J \colon M \to \mathfrak g^\ast$ of the cotangent lifted action of $G$ on $M$ is given as
\[
J(q,p) = -q_x\cdot p.
\]
\end{prop}

\begin{proof}
Using the formula for the momentum map of cotangent lifted action (see \cite[p.283]{Marsden78}) we obtain
\[
\langle v,J(q,p) \rangle
= \langle (q,p), \hat v_q \rangle
= \langle (q,p), -v\cdot q_x \rangle
= -\int_{S^1} v(x) p(x) q_x(x) \d x
= \langle v, -q_x \cdot p \rangle 
\]
as claimed.
\end{proof}

The manifold $M$ is equipped with a Poisson structure defined by the symplectic structure $\Omega$. By construction, the momentum map $J \colon M \to \mathfrak g^\ast$ is a Poisson map. It is surjective (take $q = \id$) and therefore called a {\em full realisation of $\mathfrak g^\ast$}. If $H$ is a Hamiltonian on $\mathfrak g^\ast$ then the Hamiltonian flow of the {\em collective} system $(M,\Omega,H\circ J)$ maps fibres of $J$ to fibres and descends to the Hamiltonian flow of the system $(\mathfrak g^\ast,\{\cdot,\cdot\},H)$ because the Hamiltonian vector fields are $J$-related and $J$ is a Poisson map. More generally, a symplectic map on $M$ that maps fibres to fibres descends to a Poisson map on $g^\ast$.

%

\begin{example}
As in example \ref{ex:Hong} we consider the Hamiltonian
\[H(u) = \int_{S^1} \mathcal H(u^{\mathrm{jet}}(x)) \d x\]
on $\mathfrak g^\ast$. Hamilton's equations of the collective system $(M,\Omega,H\circ J)$ are given as the following system of PDEs
\begin{align*}
q_t &= q_x \sum_{j=0}^K (-1)^j  \frac{\p^j }{\p x^j}\left( \frac{\p \mathcal H}{\p u_{x^j}}(u^{\mathrm{jet}}) \right), \\
p_t &= -\frac{\p}{\p x}\left(
p\sum_{j=0}^K (-1)^j  \frac{\p^j }{\p x^j}\left(\frac{\p \mathcal H}{\p u_{x^j}}(u^{\mathrm{jet}})\right)
\right).
\end{align*}

Choosing $\mathcal H(u) = - \frac 16 u^2$ (Burgers' equation) yields
\[
q_t = -\frac 13 q^2_xp,
\qquad
p_t = -\frac 13 (q_xp^2)_x.
\]
\end{example}

\section{Integrator of the collective system}\label{sec:collIntegrator}
\subsection{Spatial discretisation}
We use a second-order finite-difference method in space to discretise the realisation $J$ and the Hamiltonian $H$ to obtain a system of Hamiltonian ODEs in canonical form: as before, we consider $S^1$ as the quotient $\R / L\Z$. We introduce a uniform grid
$(x_1,\ldots,x_N)$, $x_j = j\cdot \Delta x$, $\Delta x = 1/N$ with $N$ points and periodic boundary conditions. Moreover, we consider the corresponding half-grid $(x_{1/2},\ldots,x_{N-1/2})$. Both grids are illustrated in figure \ref{fig:grid}.
\begin{figure}
\begin{center}
\begin{tikzpicture}[ scale=1.2]

\draw[thick, -] (-0.5,0) -- (9.5,0);

\draw (0 cm, 1pt) -- (0 cm, -1pt) node[anchor = north] {$x_{N-1}$};
\draw (1 cm, 8pt) -- (1 cm, -8pt) node[anchor = north] {$x_N=0$};

\draw (2 cm, 1pt) -- (2 cm, -1pt) node[anchor = north] {$x_{1 }$};
\draw (3 cm, 1pt) -- (3 cm, -1pt) node[anchor = north] {$x_{2}$};
\draw (4 cm, 1pt) -- (4 cm, -1pt) node[anchor = north] {$x_{3}$};

\draw (5.5 cm, 0pt) -- (5.5 cm, 0pt) node[anchor = north] {$\ldots$};
\draw (7 cm, 1pt) -- (7 cm, -1pt) node[anchor = north] {$x_{N-1}$};

\draw (8 cm, 8pt) -- (8 cm, -8pt) node[anchor = north] {$x_N=0$};
\draw (9 cm, 1pt) -- (9 cm, -1pt) node[anchor = north] {$x_{1}$};

\draw (2 cm, 8pt) -- (2 cm, 4pt) ;
\draw (3 cm, 8pt) -- (3 cm, 4pt) ;
\draw (2 cm, 6pt) -- (3 cm, 6pt) ;
\draw (2.5 cm,6pt) node[anchor = south] {$\Delta x$};

\end{tikzpicture}\\
\vspace{1cm}
\begin{tikzpicture}[ scale=1.2]

\draw[thick, -] (-0.5,0) -- (9.5,0);

\draw (0.5 cm, 1pt) -- (0.5 cm, -1pt) node[anchor = north] {$x_{N-\frac 12}$};
\draw (1 cm, 8pt) -- (1 cm, -8pt) node[anchor = north] {};

\draw (1.5 cm, 1pt) -- (1.5 cm, -1pt) node[anchor = north] {$x_{\frac 12 }$};
\draw (2.5 cm, 1pt) -- (2.5 cm, -1pt) node[anchor = north] {$x_{\frac 32}$};
\draw (3.5 cm, 1pt) -- (3.5 cm, -1pt) node[anchor = north] {$x_{\frac 52}$};

\draw (5.5 cm, 0pt) -- (5.5 cm, 0pt) node[anchor = north] {$\ldots$};
\draw (6.5 cm, 1pt) -- (6.5 cm, -1pt) node[anchor = north] {$x_{N-\frac 32}$};
\draw (7.5 cm, 1pt) -- (7.5 cm, -1pt) node[anchor = north] {$x_{N-\frac 12}$};

\draw (8 cm, 8pt) -- (8 cm, -8pt) node[anchor = north] {};
\draw (8.5 cm, 1pt) -- (8.5 cm, -1pt) node[anchor = north] {$x_{\frac 12}$};

\end{tikzpicture}

\end{center}
\caption{Uniform periodic grids on $S^1 \cong \R/L\Z$, $L>0$.}\label{fig:grid}
\end{figure}
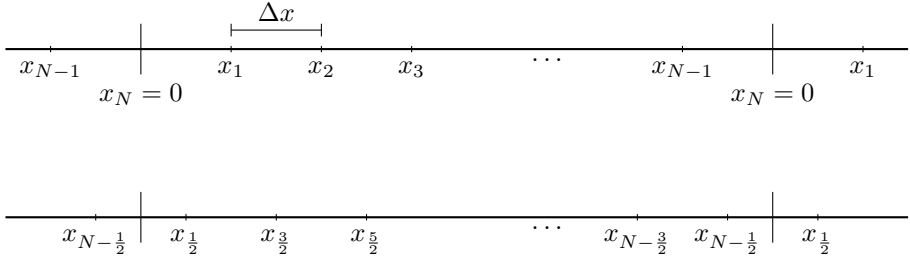
In the discretised setting, elements in $Q=\mathcal C^\infty(S^1,S^1)$ and $\mathcal C^\infty(S^1,\R)$ are approximated by their values on the considered grid. This leads to an approximation of $\mathfrak g^\ast$ and $Q$ by the vector space $\R^N$ and an approximation of $M$ by $T^\ast \R^N \cong \R^{2N}$, which we equip with coordinates $(\hat q,\hat p) = q^1,\ldots,q^N,p_1,\ldots,p_n$ in the usual way.
Discretising the symplectic structure $\Omega$ we obtain
\[
\omega = \Delta x \sum_{j=1}^N \d q^j \wedge \d p_j,
\]
which is the standard symplectic structure up to the factor $\Delta x$.
For $q \in Q$ we obtain a second-order accurate approximation $D_{\Delta x}(\hat q)$ of the spatial derivative $q_x$ on the half-grid $( 1/2 \Delta x, 3/2 \Delta x,\ldots, (N-1/2)\Delta x)$ using compact central differences as follows:
\begin{align*}
\begin{pmatrix}
q_x(\frac 12 \Delta x), q_x(\frac 32 \Delta x), \ldots,  q_x((N-\frac 32) \Delta x), q_x((N-\frac12) \Delta x)
\end{pmatrix}^T\\
\approx
\frac 1 {\Delta x} \left[
\underbrace{\begin{pmatrix}
1&0&\ldots & 0 &-1\\
-1&1&\ldots & 0 &0\\
&\ddots & \ddots \\
&&\ddots & \ddots \\
&&&-1&1
\end{pmatrix}}_{=: T}
\begin{pmatrix}
q(\Delta x)\\q(2 \Delta x)\\ \vdots \\ q((N-1) \Delta x) \\ q(N\Delta x)\end{pmatrix}
+ \begin{pmatrix}
C(q)\\0\\ \vdots \\ 0\\0
\end{pmatrix}
\right].
\end{align*}
The quantity $C(q)/L$ is the winding number (degree) of the map $q \colon S^1 \to S^1$\footnote{Let $\pi \colon \R \to S^1$ denote the universal covering of $S^1 \cong \R / L\Z$ and let $\tilde q \colon \R \to \R$ be any lift of the map $\pi \circ q \colon \R \to S^1$ to the covering space. Now $C(q) = \tilde q(L) - \tilde q(0)$. If, for instance, $q$ is the identity map on $S^1$ then $C(q)=L$.}. The values for $q_x$ are now available on the half-grid. Notice that the quantity $C(q)$ is constant if $q$ evolves smoothly subject to the PDE \eqref{eq:HamCanonicalLooking} because $C(q)$ can only take values in $L\Z$. 

A discrete version of the map $J \colon M \to g^\ast$ is given by $\hat J \colon \R^{2N}\to \R^N$ with $\hat J(\hat q, \hat p) = D_{\Delta x}\hat q .S \hat p$. Its values correspond to the half-grid.
The matrix $S$ is given as
\[
S = \frac 12 \begin{pmatrix}1&&&1\\ 1&1\\ &\ddots&\ddots\\ &&1&1 \end{pmatrix}.
\]
It averages the values of $\hat p$ to obtain second order accurate approximations of $p$ on the half-grid. In this way, we obtain approximations to $u = q_xp$ on the half grid. Approximations for $u_x$ and higher derivatives are obtained by successively applying $T_{\Delta x}$ and $T_{\Delta x}^{T}$, i.e.\
\begin{equation}\label{eq:JetU}
\frac{\p^k u}{\p x^k}
\approx \frac{\p^k_{\Delta x} u}{\p_{\Delta x} x^k}
:=
\begin{cases}
\qquad D_{\Delta x}\hat q .S \hat p \quad &\text{ if } k=0\\[0.8 em]
-T^T \frac{\p_{\Delta x}^{k-1} u}{\p_{\Delta x} x^{k-1}}/\Delta x \quad &\text{ if $k$ is odd }\\[0.8 em]
\;\;\,T\;\, \frac{\p_{\Delta x}^{k-1} u}{\p_{\Delta x} x^{k-1}}/\Delta x \quad &\text{ if $k$ is even.}
\end{cases}
\end{equation}
Here $T^T$ denotes the transpose of the matrix $T$ and $.$ denotes component-wise multiplication. Now all approximations for even derivatives are available on the half-grid and all odd derivatives on the full-grid.
A Hamiltonian of the form $\int_{S^1} \mathcal H(u,u_x,u_{xx},\ldots) \d x$ is approximated by the sum
\begin{equation}\label{eq:ApproxInH}
\int_{S^1} \mathcal H(u,u_x,u_{xx},\ldots) \d x \approx \Delta x \sum_{j=1}^{N} \mathcal H(u(x_{j-1/2}),u_x(x_{j-1/2}),u_{xx}(x_{j-1/2}),\ldots).
\end{equation}
To evaluate \eqref{eq:ApproxInH}, all approximations of $\frac{\p^k u}{\p x^k}$ where $k$ is odd are multiplied by $S$ such that the approximation of the jet of $u$ is available on the half-grid. The second-order averaging with $S$ can be avoided if $\mathcal H$ is of the form
\begin{align*}
\mathcal H (u^{\mathrm{jet}}(x))
&= \mathcal H^{\mathrm{even}} (u(x),u_{xx}(x),u_{xxxx}(x),\ldots)\\
&+\mathcal H^{\mathrm{odd}} (u_x(x),u_{xxx}(x),u_{xxxxx}(x),\ldots).
\end{align*}
We can then approximate the Hamiltonian by
\begin{align}\nonumber
\int_{S^1} \mathcal H(u,u_x,u_{xx},\ldots) \d x 
&\approx
  \Delta x \sum_{j=1}^{N} \mathcal H^{\mathrm{even}}(u(x_{j-1/2}),u_{xx}(x_{j-1/2}),u_{xxxx}(x_{j-1/2})\ldots)\\ \label{eq:HApproxSplit}
&+\Delta x \sum_{j=1}^{N} \mathcal H^{\mathrm{odd}}(u_x(x_{j}),u_{xxx}(x_{j}),u_{xxxxx}(x_{j}),\ldots).
\end{align}

Taking into account that the symplectic form $\omega$ is the canonical symplectic structure scaled by $\Delta x$, defining $\hat H$ as
\begin{equation}\label{eq:hatH}
\hat H(u)= \sum_{j=1}^{N} \mathcal H(u(x_{j-1/2}),u_x(x_{j-1/2}),u_{xx}(x_{j-1/2}),\ldots) 
\end{equation}
or as the corresponding term from \eqref{eq:HApproxSplit} puts Hamilton's equations into the canonical form
\begin{equation}\label{eq:HamCanonical}
\dot {\hat q} = \;\;\, \nabla_{\hat p} \bar {\hat H}(\hat q, \hat p),
\qquad \dot {\hat p} = -\nabla_{\hat q} \bar {\hat H}(\hat q, \hat p)
\end{equation}
with collective Hamiltonian $\bar {\hat H} = \hat H \circ \hat J \colon \R^{2N}\to \R$.
Here the dot denotes the time-derivative.
Finally, \eqref{eq:HamCanonical} is a 2nd order accurate, spatial discretisation of \eqref{eq:HamCanonicalLooking}.

\begin{remark}
An alternative to the described finite-difference discretisation are spectral methods. Notice that $q \in \mathcal C^\infty(S^1,S^1)$ can be split into the winding term $C(q)\id$ and the term $q-C(q)\id$ which has winding number zero. In a pseudo-spectral discretisation, the derivative of $q-C(q)\id$ is calculated in a Fourier basis and the winding term $C(q)\id$ is accounted for in the derivative $q_x$ by adding the constant $C(q)/L$ component-wise. The derivatives of $u=q_xp$ can be calculated without complications.

A full spectral discretisation is also possible because embedding $\mathcal C^\infty(S^1,S^1)$ and $\mathcal C^\infty(S^1,\R)$ into the Hilbert space $L_2$ and choosing any orthonormal basis will lead to a symplectic form $\omega$ which is in the standard form (splitting $q$ as above to allow for a Fourier basis). Therefore, Hamilton's equations for the basis coefficients appear in canonical form.
\end{remark}

\subsection{The integration scheme}

A numerical solution to the original equation \eqref{eq:Hamongstar} can now be obtained as follows.
\begin{enumerate}
\item
Lift an initial condition \[\hat u ^{(0)}=(u^{(0)}(x_1),\ldots,(u^{(0)}(x_N))\] to $(\hat q^{(0)},\hat p^{(0)})  \in \hat J^{-1}(\hat u ^{(0)})$, for example by setting 
\begin{align*}
\hat q^{(0)} &= (\Delta x, 2\Delta x,\ldots, N\Delta x),\\
\hat p^{(0)} &=\hat u ^{(0)},
\end{align*}
as we will do in our numerical experiments. Notice that $\hat q^{(0)}$ is a discretisation of the identity map on $S^1$. The exact and discrete derivative is the constant 1 function or vector.
\item
The system of Hamiltonian ODEs \eqref{eq:HamCanonical} can be integrated subject to the initial conditions $(\hat q^{(0)},\hat p^{(0)})$ using a symplectic numerical integrator. 
\item
Approximations to $u$ can be calculated from $(\hat q,\hat p)$ on the half-grid as $D_{\Delta x} \hat q . S \hat p$. 
\end{enumerate}

\begin{remark}\label{rem:energyDiscreteCollective}
Conservation of $\bar{\hat H}$ in \eqref{eq:HamCanonical} exactly corresponds to conservation of the discretised Hamiltonian $\hat H$ \eqref{eq:ApproxInH} or \eqref{eq:HApproxSplit} because we consistently relate $u$ and $(q,p)$ by \eqref{eq:JetU}. Therefore, using a symplectic integrator to solve the system \eqref{eq:HamCanonical} of Hamiltonian ODEs we expect excellent energy behaviour of the numerical solution.
In the following numerical experiments we will use the symplectic implicit midpoint rule. The arising implicit equations will be solved using Newton iterations.

\end{remark}

\begin{remark}
In contrast to the case of Hamiltonian-ODEs on Poisson manifolds, it is hard for a symplectic integrator to maintain the structure fibration on the symplectic manifolds induced by the discretisation $\hat J$ of the realisation $J$. Indeed, the implicit midpoint rule used in our numerical examples fails to do so. This is why we do {\em not} obtain a (discretisation of a) Poisson integrator in this way. However, the described energy conservation properties of remark \ref{rem:energyDiscreteCollective} are independent of this drawback. Moreover, our numerical examples will show that we obtain excellent Casimir behaviour although this has not been forced by this construction. 
\end{remark}

\section{Numerical experiments}
\graphicspath{{../figs/}}

For the following numerical experiments, we consider Hamiltonian systems $(\mathrm{diff}^*(S^1),\{\cdot,\cdot\},H)$ with
\begin{equation}\label{Ham}
H = \int_{\mathcal{S}^1}\left(C_1 u^2 + C_2 u_x^2 + C_3 u^3 + C_4 u_x^3\right) \mathrm{d}x.
\end{equation}

To gain a sense of the relative performance of the collective integration method from section \ref{sec:collIntegrator} we will now develop a conventional finite-difference approach for comparison that is based on \cite{McL03}. \\

First, a finite-dimensional discrete Hamiltonian approximation is obtained by 
\begin{equation}
\hat{H} = \Delta x\sum_{j=1}^{N}(C_1 \hat{u}_j^2 + C_2 (\hat{u}_x)_j^2 + C_3 \hat{u}_j^3 + C_4 \left(\hat{u}_x\right)_j^3),
\end{equation}
where $\hat{u}_x = T\hat{u}/\Delta x$ is a compact finite-difference approximation. The PDE is then written as a set of the Hamiltonian ODEs in skew-gradient form 
\begin{equation}\label{uhatdot}
	\dot{\hat{u}} = K(\hat{u})\nabla_{\hat{u}}\hat{H}_{\Delta x}.
\end{equation}
Here, $K(\hat{u}) = (UD^{(1)}+D^{(1)}U)$ represents the discrete version of the coadjoint operator in equation \eqref{eq:Hamongstar}, where $U=\mathrm{diag}(\hat{u})$ is a diagonal matrix with $\hat{u}_j$ on the $j$th diagonal and the matrix $D^{(1)}$ is a centered finite-difference matrix with the stencil $[-\frac{1}{2\Delta x},0,\frac{1}{2\Delta x}]$ on the main three diagonals and $-\frac{1}{2\Delta x}$ and $\frac{1}{2\Delta x}$ on the top right and bottom left corners, respectively. This yields a skew-symmetric tri-diagonal matrix $K(\hat{u})$ given as
\begin{equation*}
	\frac{1}{2 \Delta x}\left(\begin{array}{ccccc}
	0&u_1+u_2&&&-u_{n}-u_1\\
	-u_1-u_2&0&u_2+u_3&  &\\
	&\ddots & \ddots & \ddots&\\
	&&-u_{n-2}-u_{n-1}&0&u_{n-1}+u_n\\
	u_1+u_n&&&-u_{n-1}-u_n&0
	\end{array}\right),
\end{equation*}
where the diagonal dots denote the continuation of the stencil $[-u_{i-1}-u_{i},0,u_{i}+u_{i+1}]$ on the $i$th row. Note that
\begin{equation}
	\frac{\mathrm{d}}{\mathrm{d}t}\hat{H} = (\nabla_{\hat{u}}\hat{H})^{\mathrm{T}}\dot{\hat{u}} = (\nabla_{\hat{u}}\hat{H})^{\mathrm{T}} K(\hat{u})\nabla_{\hat{u}}\hat{H} = 0,
\end{equation}
hence, $\hat{H}$ is a first integral of this ODE. Finally, equation \eqref{uhatdot} is integrated using the implicit midpoint rule, which is solved using Newton iterations. This method will henceforth be referred to as the \textit{conventional method}. \\

The conventional and collective methods are both order-two in space as shown by figure \ref{fig:Convergence}, which show errors for travelling wave solutions of the cubic Hamiltonian system outlined in section \ref{sec:P}. The Hamiltonian error at time $t=t_n$ is calculated by $(\hat{H}(0)-\hat{H}(t_n))/\hat{H}(0)$ and similarly for the Casimir error. The solution error is 
\begin{equation*}
\frac{||\hat{u}_n-\hat{u}_e||_2}{||\hat{u}_e||_2},
\end{equation*}
where $\hat{u}_n$ is the numerical solution, $\hat{u}_e$ is the exact solution evaluated on the grid and $||\cdot||_2$ is the discrete $L_2$-norm. We see from figure \ref{fig:Convergence_H} that the collective method preserves the energy up to machine precision for this experiment. We remark that the solution error observed in figure \ref{fig:Convergence_E} is largely attributed to phase error and does not reflect the ability of the method to preserve the shape of the travelling wave.

\begin{figure}
	\centering
	\begin{subfigure}{0.32\textwidth}
		\centering
		\includegraphics[width=\textwidth]{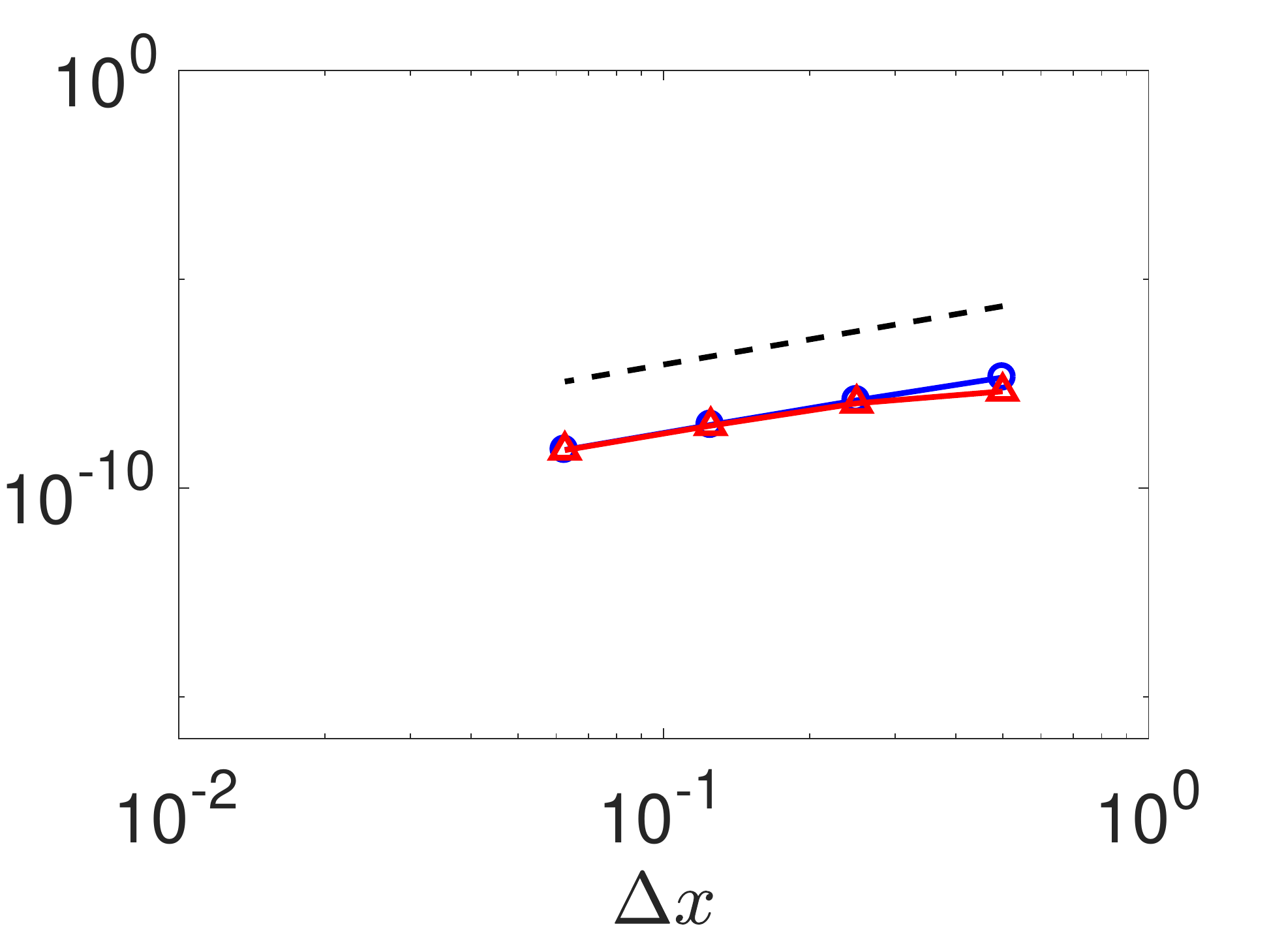}
		\subcaption{Casimir error}
		\label{fig:Convergence_C}
	\end{subfigure}
	\begin{subfigure}{0.32\textwidth}
		\centering
		\includegraphics[width=\textwidth]{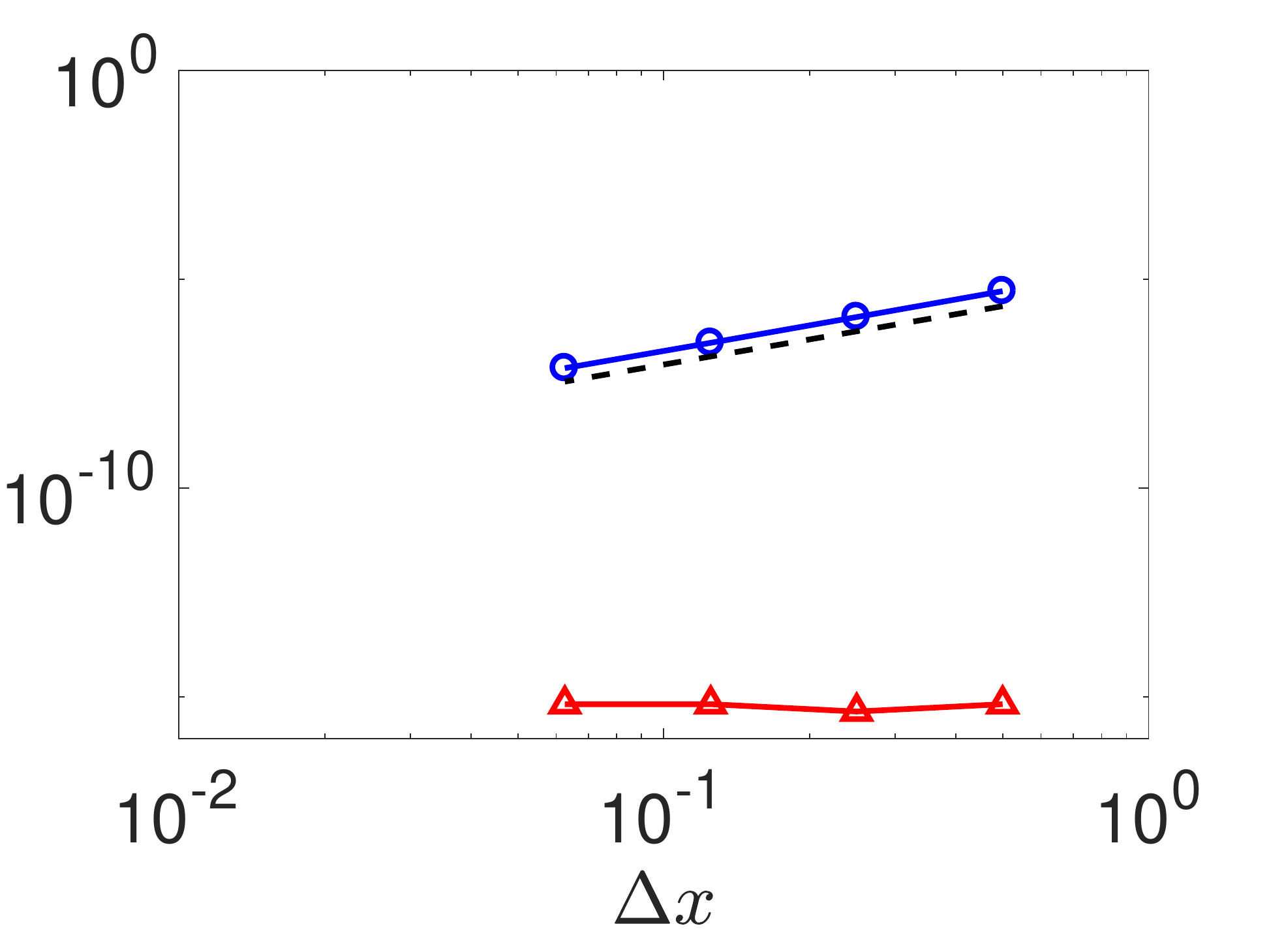}
		\subcaption{Hamiltonian error}
		\label{fig:Convergence_H}
	\end{subfigure}
	\begin{subfigure}{0.32\textwidth}
		\centering
		\includegraphics[width=\textwidth]{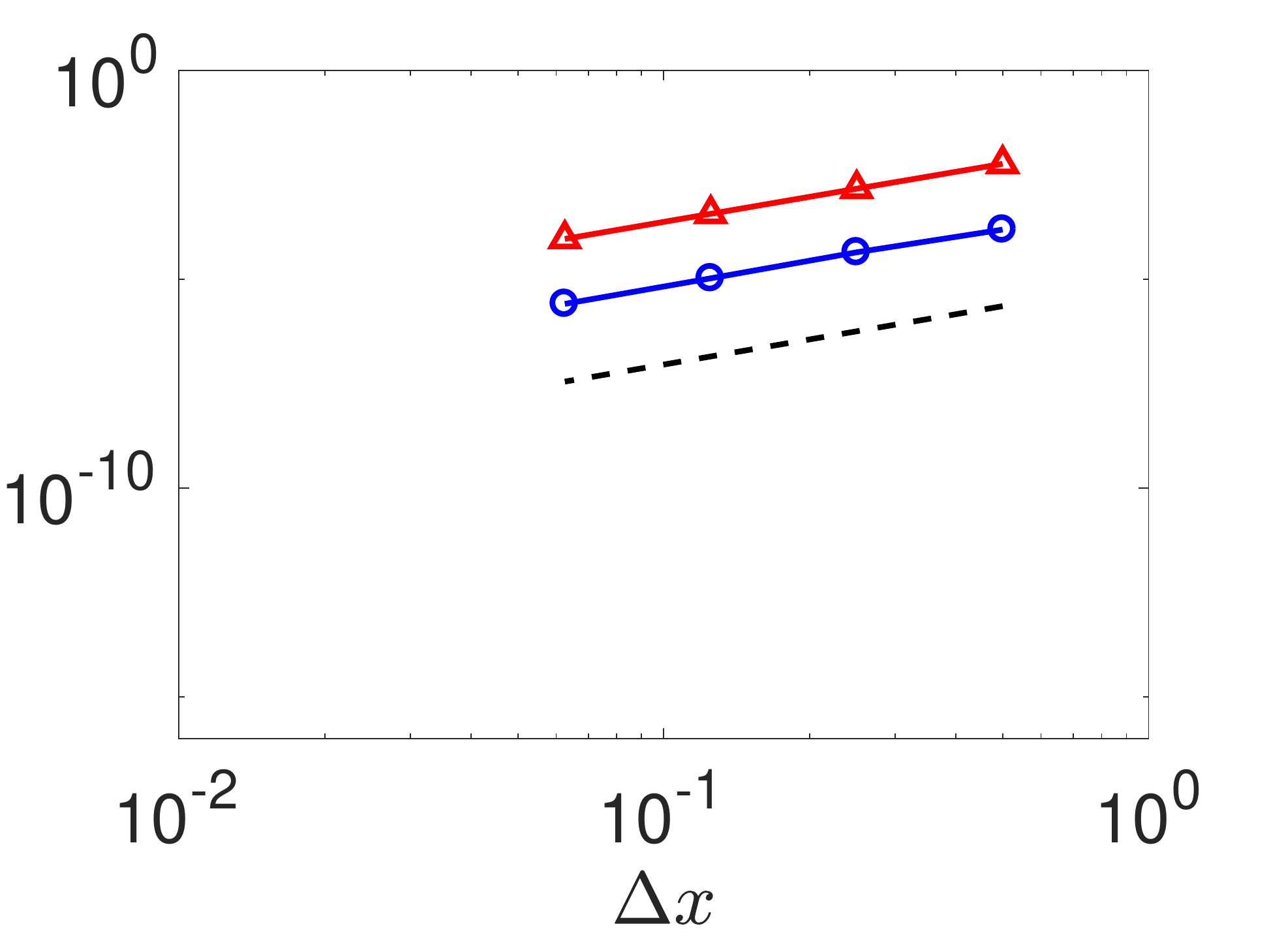}
		\subcaption{Solution error}
		\label{fig:Convergence_E}
	\end{subfigure}
	\caption{Order-two convergence for the travelling wave solution of the extended Burgers' equation outlined in section \ref{sec:P}. The plots correspond to the conventional solution (\marksymbol{o}{blue}) and the collective solution (\marksymbol{triangle}{red}) and an order-two reference line (\RefLine). The error is calculated after 512 timesteps, with $L = 8$, $\Delta t = 2^{-14}$ and $\Delta x = L/2^{k}$ for $k=1,2,3$ and $4$. }\label{fig:Convergence}
\end{figure}

\subsection{Inviscid Burgers' equation}
Setting $C_1 = 1$, $C_2 = 0$, $C_3 = 0$ and $C_4 = 0$ in equation \eqref{Ham} yields the well-known inviscid Burgers' equation
\begin{equation*}
u_t = 6 u u_x.
\end{equation*}
In the following example, the equation is modelled with the initial conditions $u(0,x) = 1 + \frac{1}{2} \cos(2 \pi x /L)$, which develops a shock wave at about $t=0.4$. Figure \ref{fig:Burgers} shows three snapshots of the conventional and the collective solutions before and after the shock and figure \ref{fig:Burgers_errors} shows the Casimir and Hamiltonian errors over time. Over the short simulation time, both methods yield qualitatively similar solutions and it is difficult to tell them apart. Due to the presence of shock waves in the inviscid Burgers' equation, it is difficult to gain a sense of the long term behaviour of the methods as no solution exists after a finite time. From figure \ref{fig:buregerHerror} we see that the conventional method has exceptional Hamiltonian preservation properties and maintains the error at machine precision throughout the simulation. This can be explained by the fact that the implicit midpoint rule preserves quadratic invariants, that is, $\hat{H}$ is preserved exactly by the conventional method. Otherwise, the errors grow quadratically until the shock develops, after which, they appear bounded. The Hamiltonian error of the collective solution can also be reduced to machine precision by reducing the time step $\Delta t$.
\begin{figure}
	\centering
	\begin{subfigure}{0.3\textwidth}
		\centering
		\includegraphics[width=\textwidth]{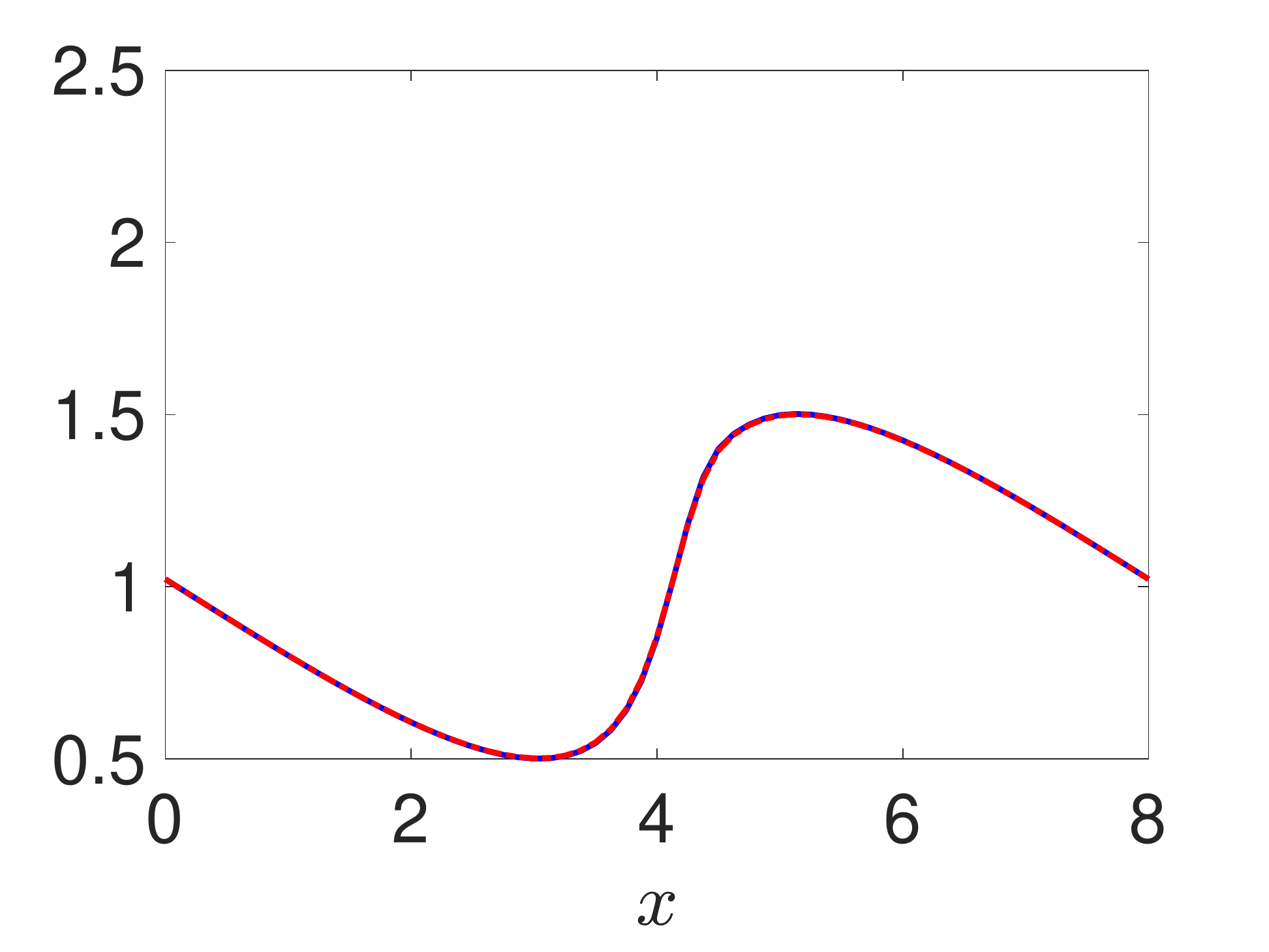}
		\subcaption{$t=0.317$}
		\label{}
	\end{subfigure}
	\begin{subfigure}{0.3\textwidth}
		\centering
		\includegraphics[width=\textwidth]{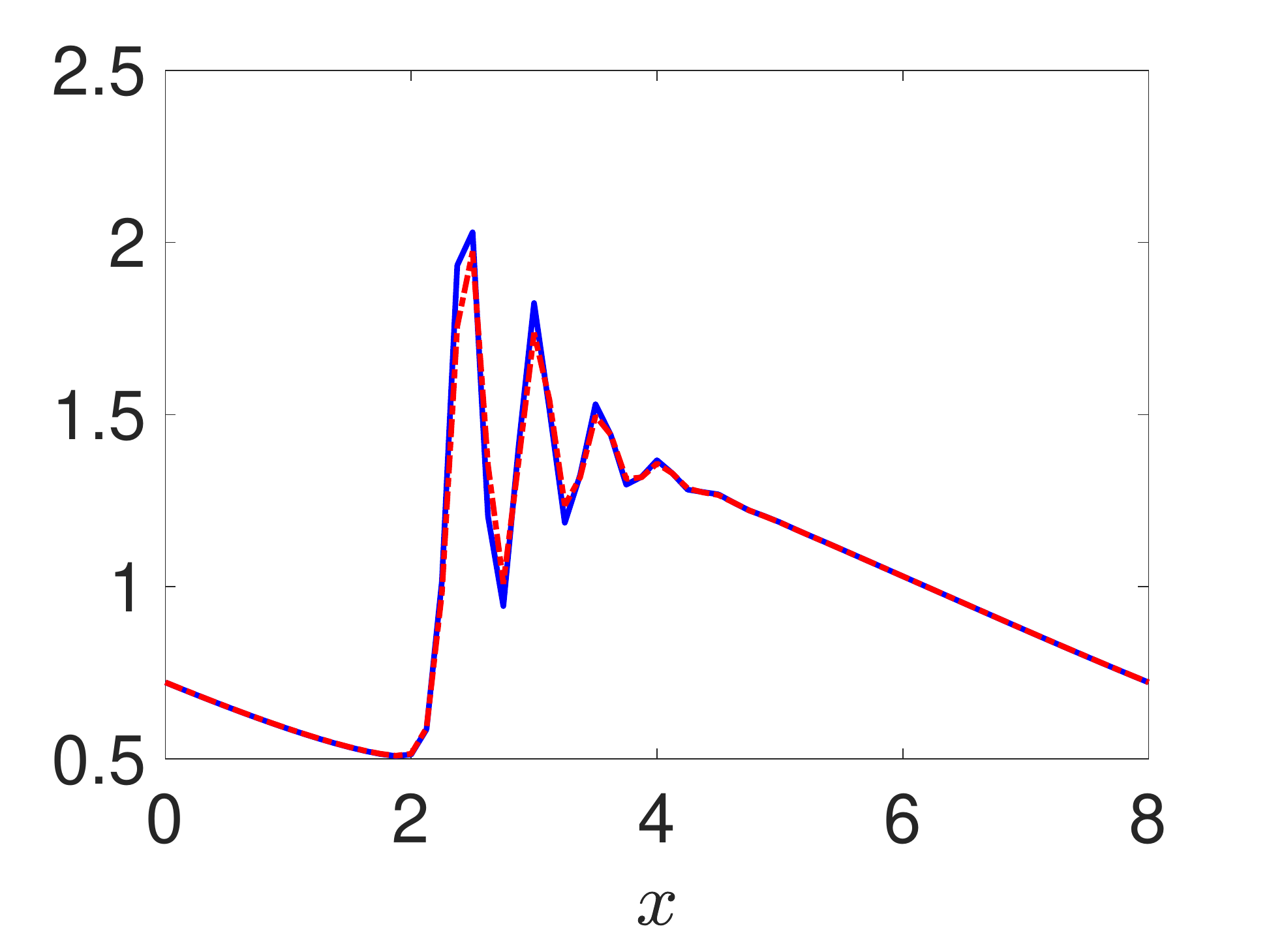}
		\subcaption{ $t=0.635$}
		\label{}
	\end{subfigure}
	\begin{subfigure}{0.3\textwidth}
		\centering
		\includegraphics[width=\textwidth]{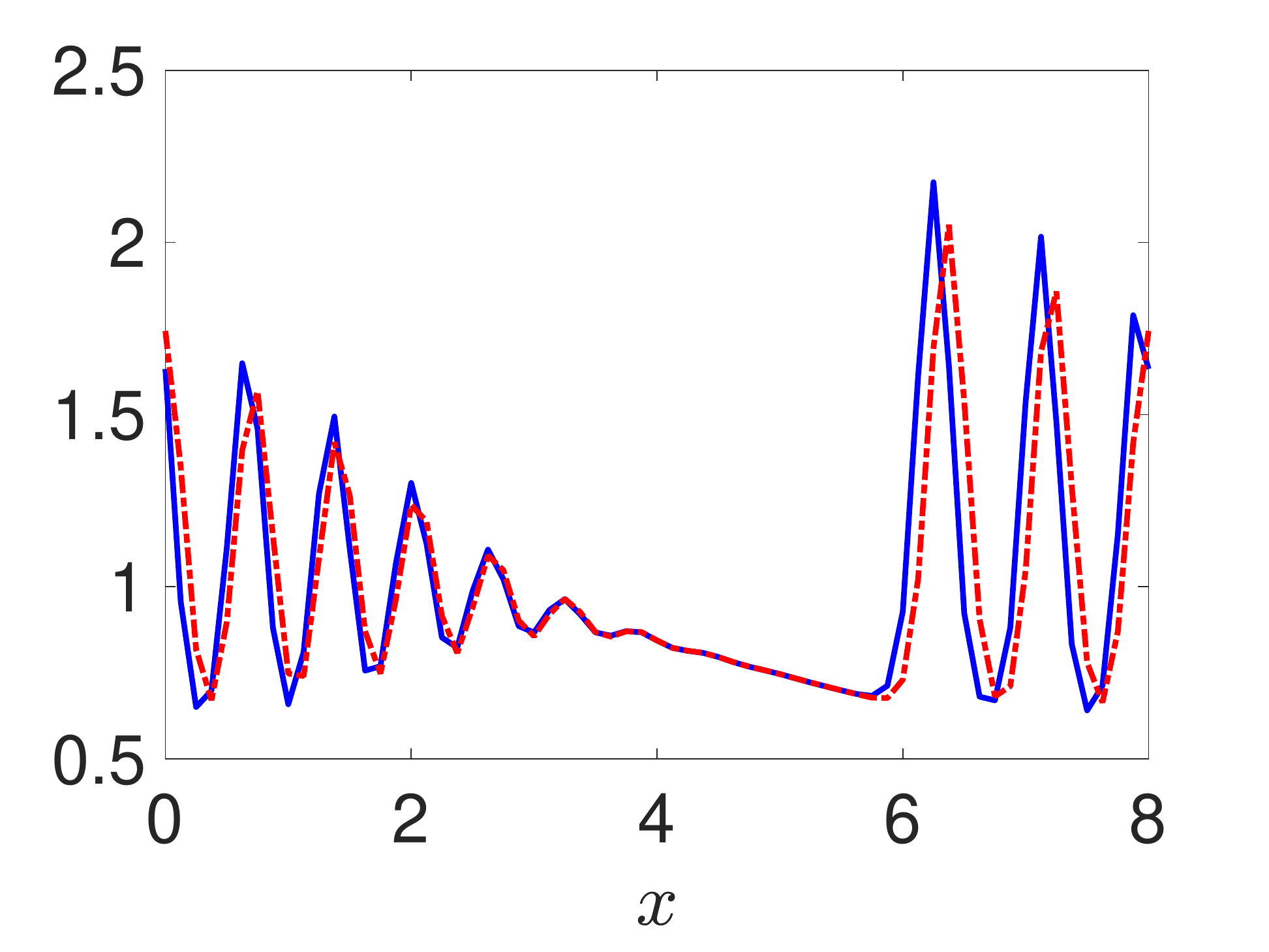}
		\subcaption{$t=1.37$}
		\label{}
	\end{subfigure}
	\caption{ Inviscid Burgers' equation solutions of the conventional method (\ConvLine) and collective method (\CollLine). The grid parameters are $n_x = 64$, $\Delta x = 0.125$, $L = 8$ and $\Delta t = 2^{-12}$. A shock forms at about $t=0.4$.}
	\label{fig:Burgers}
\end{figure}
\begin{figure}
	\centering
	\begin{subfigure}{0.32\textwidth}
		\centering
		\includegraphics[width=\textwidth]{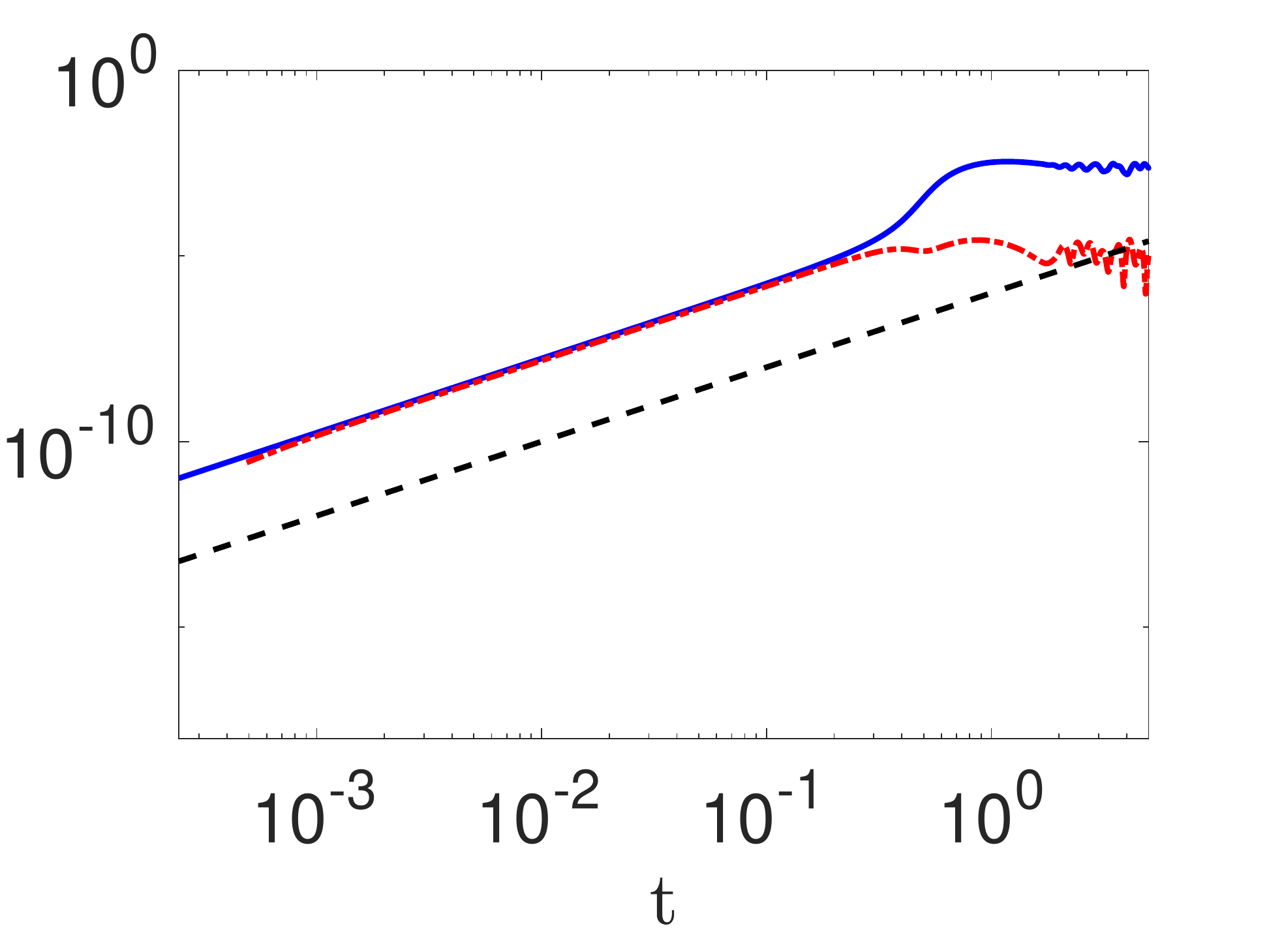}
		\subcaption{Casimir error}
		\label{}
	\end{subfigure}
	\begin{subfigure}{0.32\textwidth}
		\centering
		\includegraphics[width=\textwidth]{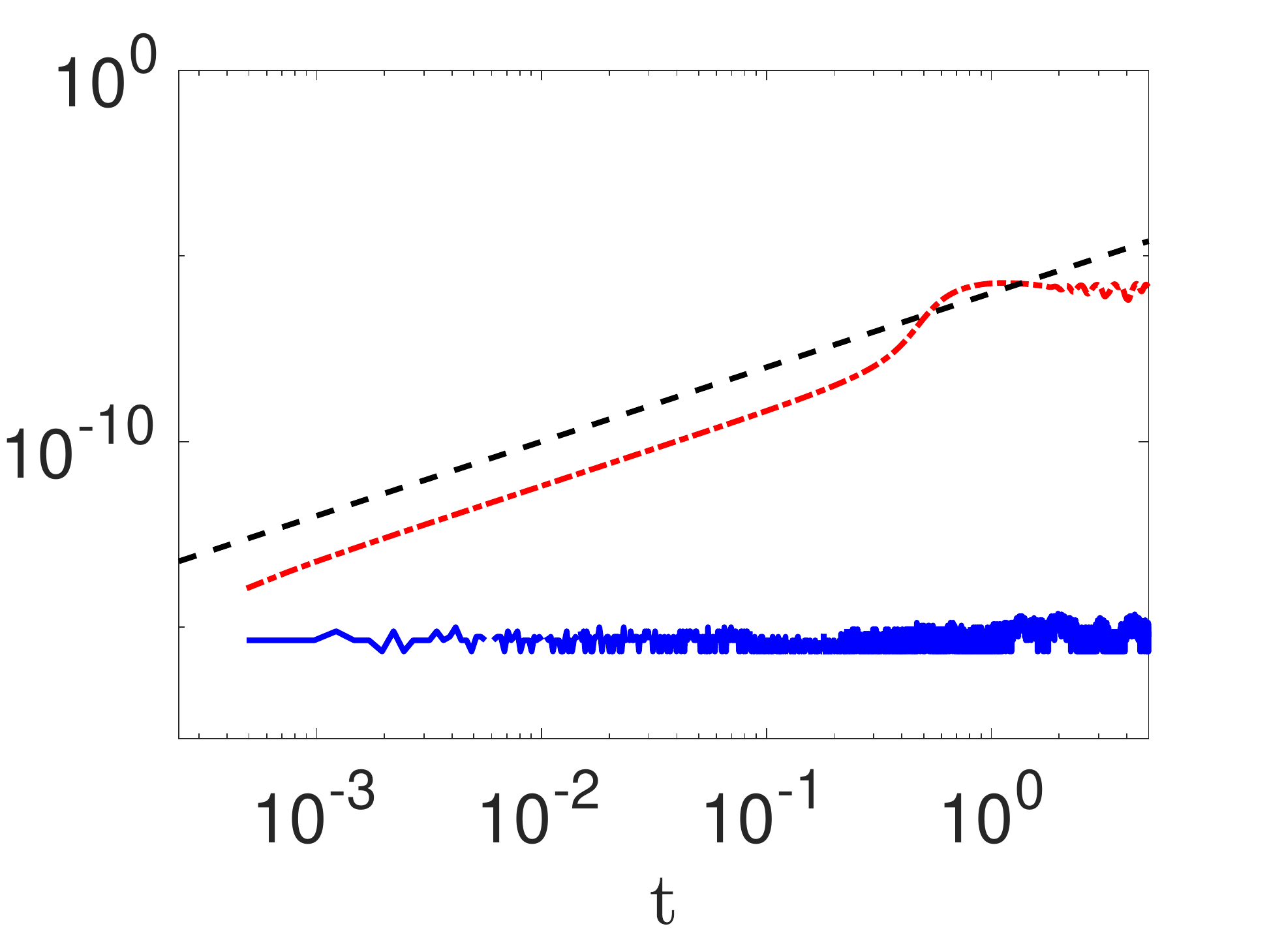}
		\subcaption{Hamiltonian error}
		\label{fig:buregerHerror}
	\end{subfigure}
	%
	\caption{The errors corresponding to the conventional (\ConvLine) and collective (\CollLine) methods for the inviscid Burgers' equation and $\mathcal{O}(t^2)$ reference lines (\RefLine). }\label{fig:Burgers_errors}
\end{figure}
\subsection{Extended Burgers' equation}\label{sec:P}
We now focus our attention to a cubic Hamiltonian problem that we have designed to admit non-symmetric travelling wave solutions. The PDE being modelled arises from setting $C_1 = 1/2$, $C_2 = 1/2$, $C_3 = -1/4$ and $C_4 = 1/2$ in equation \eqref{Ham}, which yields
\begin{equation*}
		u_t = 3uu_x - \frac{9}{4}u^2u_x - u_xu_{xx} - 3u_x^2u_{xx} - 2uu_{xxx} - 2uu_xu_{xxx} - 6uu_{xx}^2
\end{equation*}
and is henceforth referred to as the \textit{extended Burgers' equation}.
\subsubsection{Travelling wave solutions}
In this example, we look for solutions of the form $u(x,t)=f(s)$, where $s=x-ct$ for wave velocity $c$. This yields an ODE in $s$, which is solved to a high degree of accuracy on the grid using MATLAB's $\texttt{ode45}$. Figure \ref{fig:TW_P} shows snapshots of travelling wave solutions to the extended inviscid Burgers' equation and their Fourier transforms and figure \ref{fig:errors_TW} shows the corresponding errors. The main observations concerning these figures is that the errors of the collective solution are bounded whereas the conventional solution errors grow with time. In particular, the high frequency Fourier modes of the conventional solution erroneously drift away from that of the exact solution while the collective solution does a reasonably good job at keeping these modes bounded. These erroneously large high frequency modes can be seen with the naked eye in figure \ref{tw3}. This is again highlighted by figure \ref{fig:FM_P}, which shows that the highest frequency mode (i.e., the mode whose wavelength is equal to the grid spacing $\Delta x$) grows exponentially in time. Figures \ref{fig:errors_TW_C} and \ref{fig:errors_TW_H} show the behaviour of the Casimir and Hamiltonian errors. This highlights the ability of the collective method to keep the errors bounded, while the errors of the conventional solution grow linearly with time. Towards the end of the simulation, the errors of the conventional solution become so large that the implicit equations arising from the midpoint rule become too difficult to solve numerically and the Newton iterations fail to converge. The simulation ends with the conventional method errors diverging to infinity. 
\begin{figure}
	\centering
	\begin{subfigure}{0.32\textwidth}
		\centering
		\includegraphics[width=\textwidth]{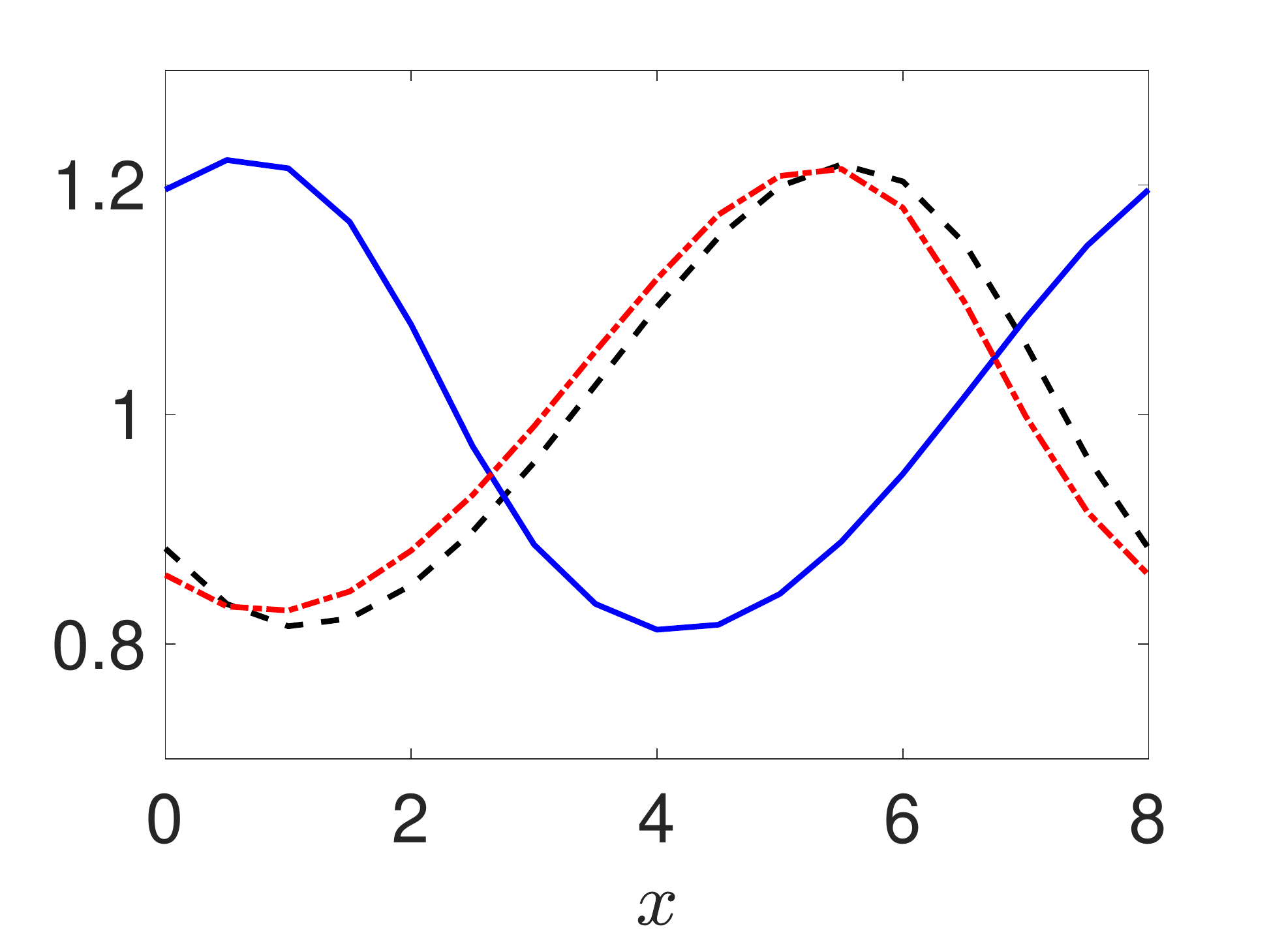}
		\subcaption{}
		\label{tw1}
	\end{subfigure}
	\begin{subfigure}{0.32\textwidth}
		\centering
		\includegraphics[width=\textwidth]{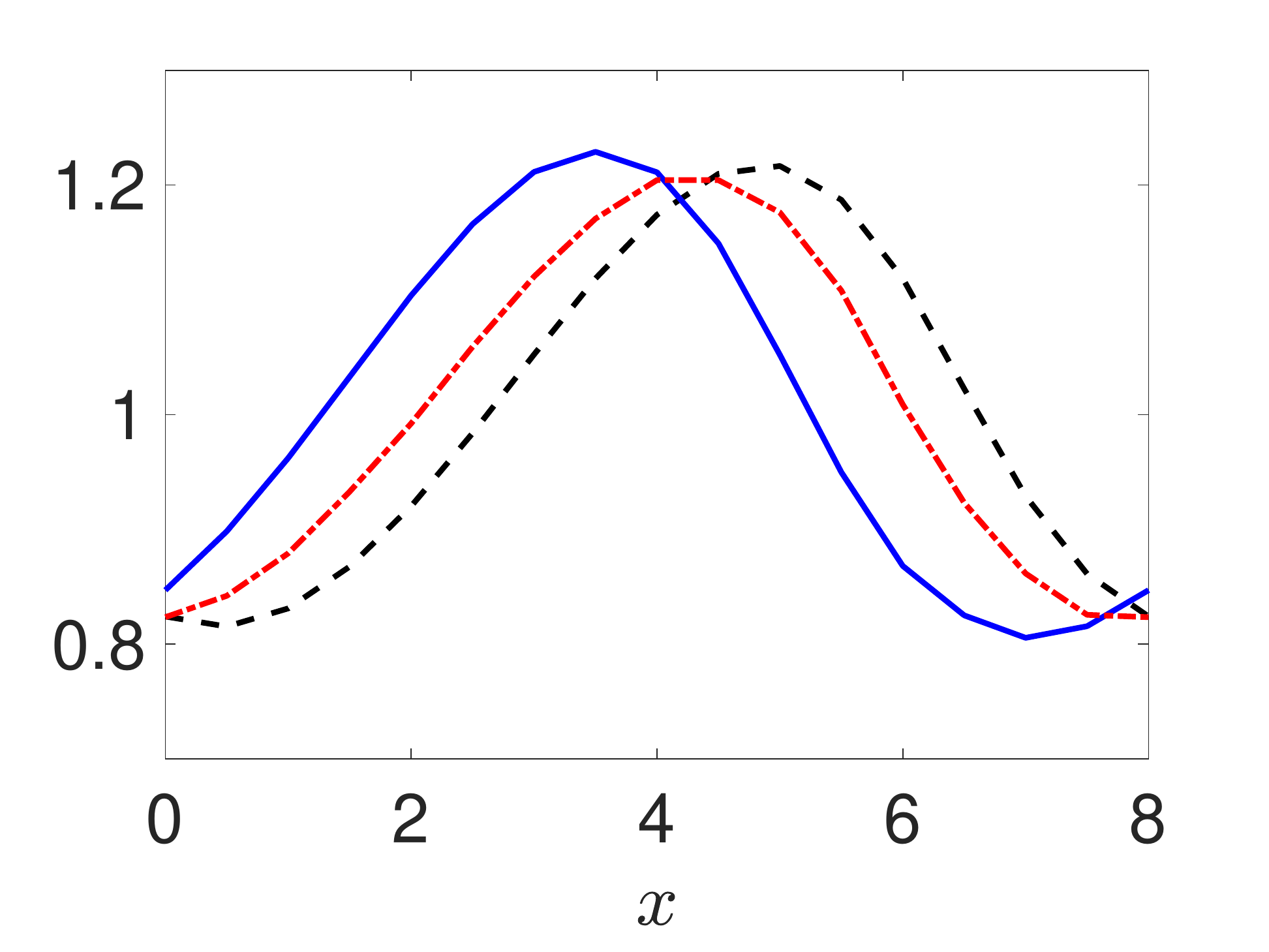}
		\subcaption{}
		\label{}
	\end{subfigure}
	\begin{subfigure}{0.32\textwidth}
		\centering
		\includegraphics[width=\textwidth]{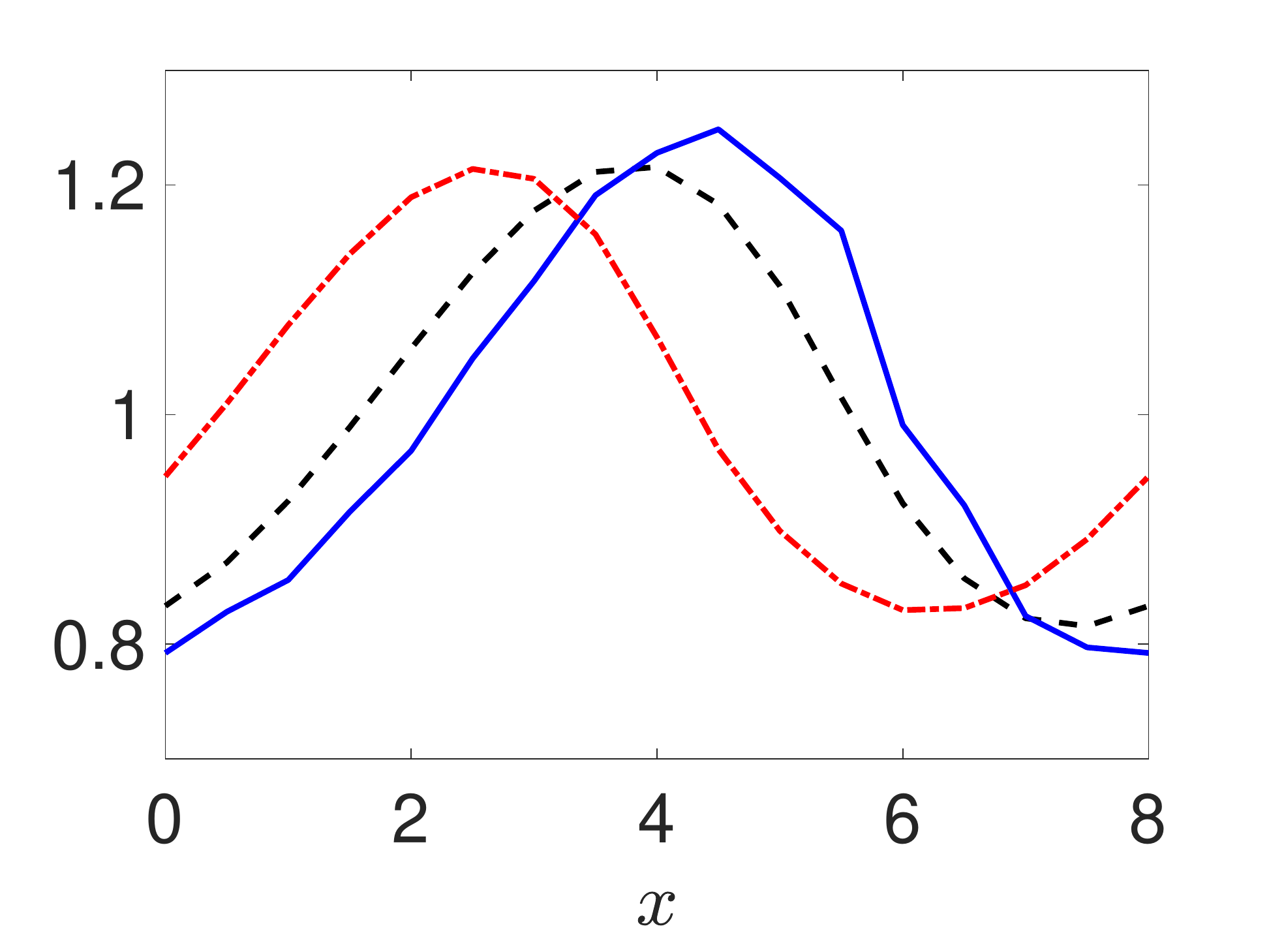}
		\subcaption{}
		\label{tw3}
	\end{subfigure}
	
	\begin{subfigure}{0.32\textwidth}
		\centering
		\includegraphics[width=\textwidth]{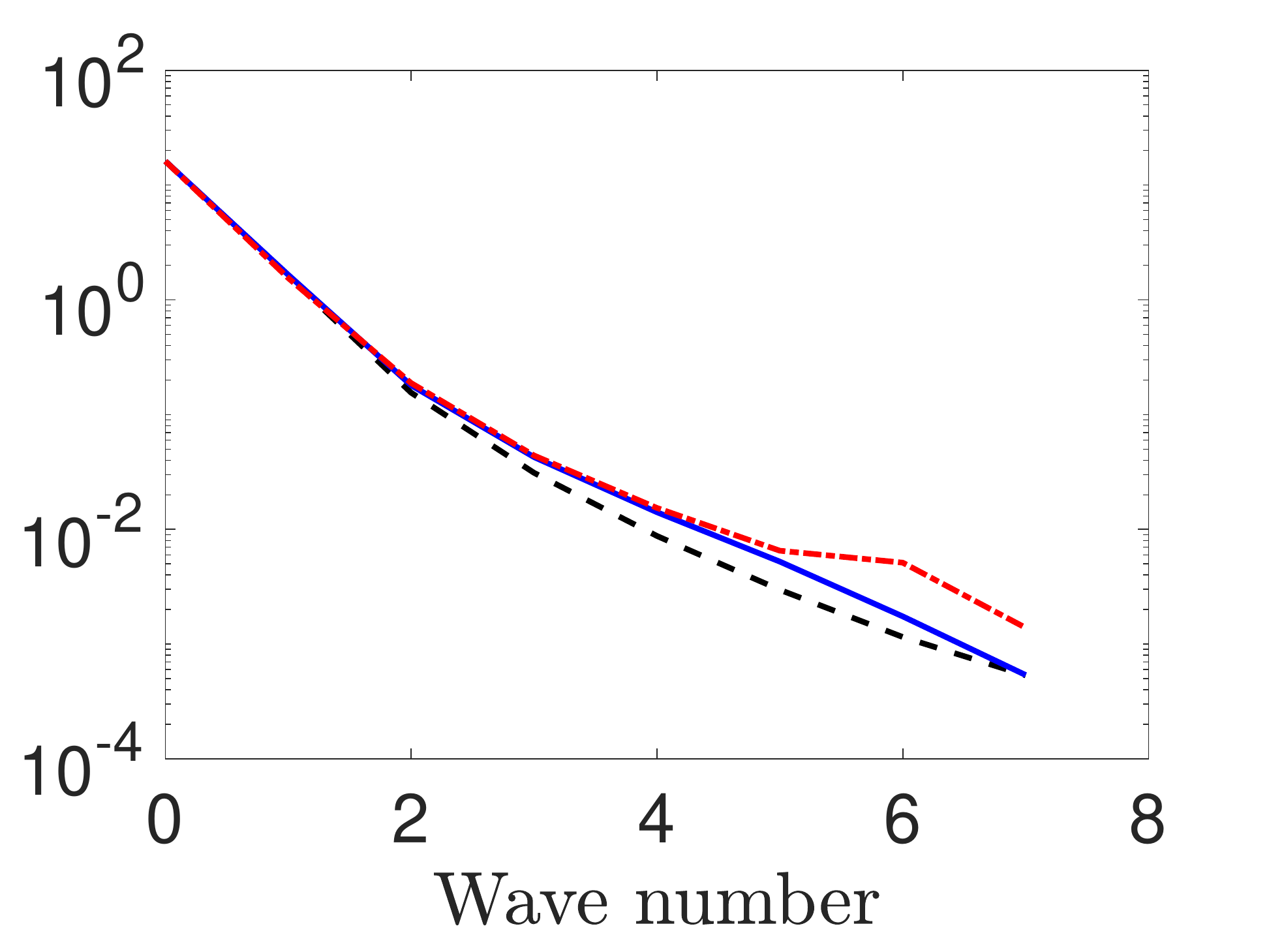}
		\subcaption{}
		\label{}
	\end{subfigure}
	\begin{subfigure}{0.32\textwidth}
		\centering
		\includegraphics[width=\textwidth]{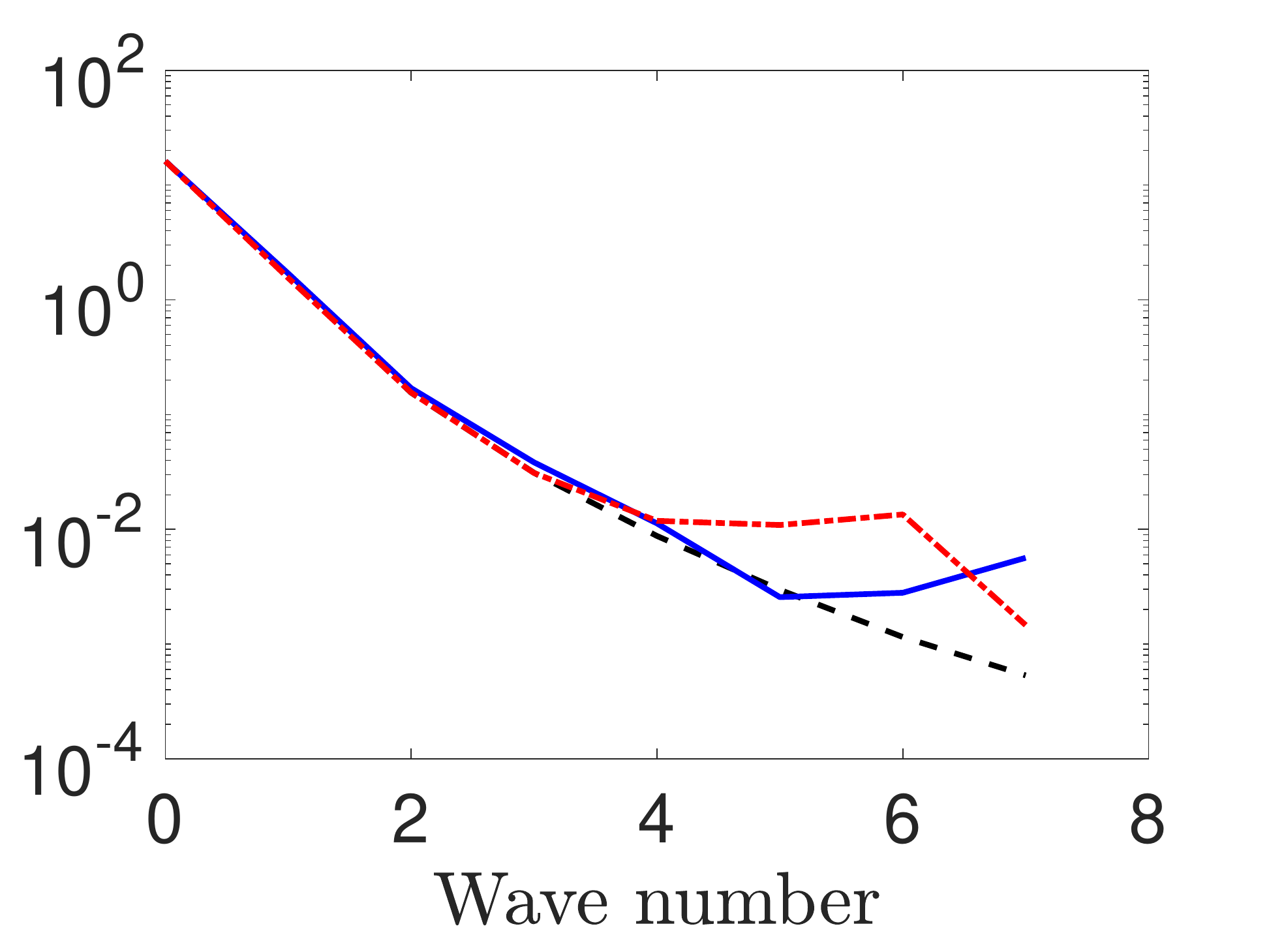}
		\subcaption{}
		\label{}
	\end{subfigure}
	\begin{subfigure}{0.32\textwidth}
		\centering
		\includegraphics[width=\textwidth]{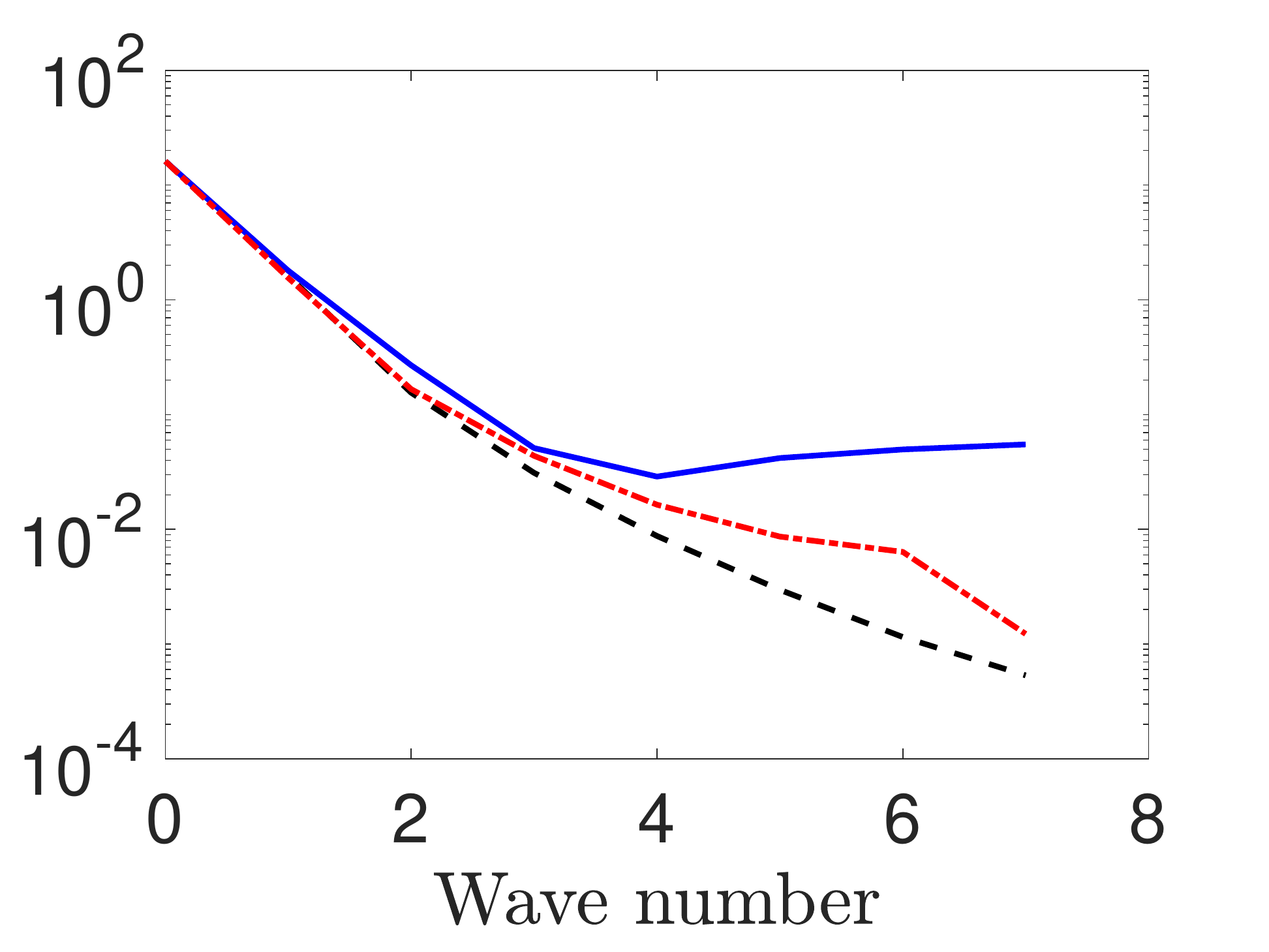}
		\subcaption{}
		\label{}
	\end{subfigure}
	\caption{Travelling wave solutions of the perturbed Burgers' equation (top row) and the positive Fourier modes (bottom row) at $t=109$ (left column), $t=218$ (middle column) and $t=437$ (right column). The plots correspond to the conventional method (\ConvLine), collective method (\CollLine) and the exact travelling wave solution (\RefLine). The grid parameters are $n_x = 16$, $\Delta x = 0.5$, $L = 8$ and $\Delta t = 2^{-6}$.}\label{fig:TW_P}
\end{figure}

\begin{figure}
	\centering
	\begin{subfigure}{0.32\textwidth}
		\centering
		\includegraphics[width=\textwidth]{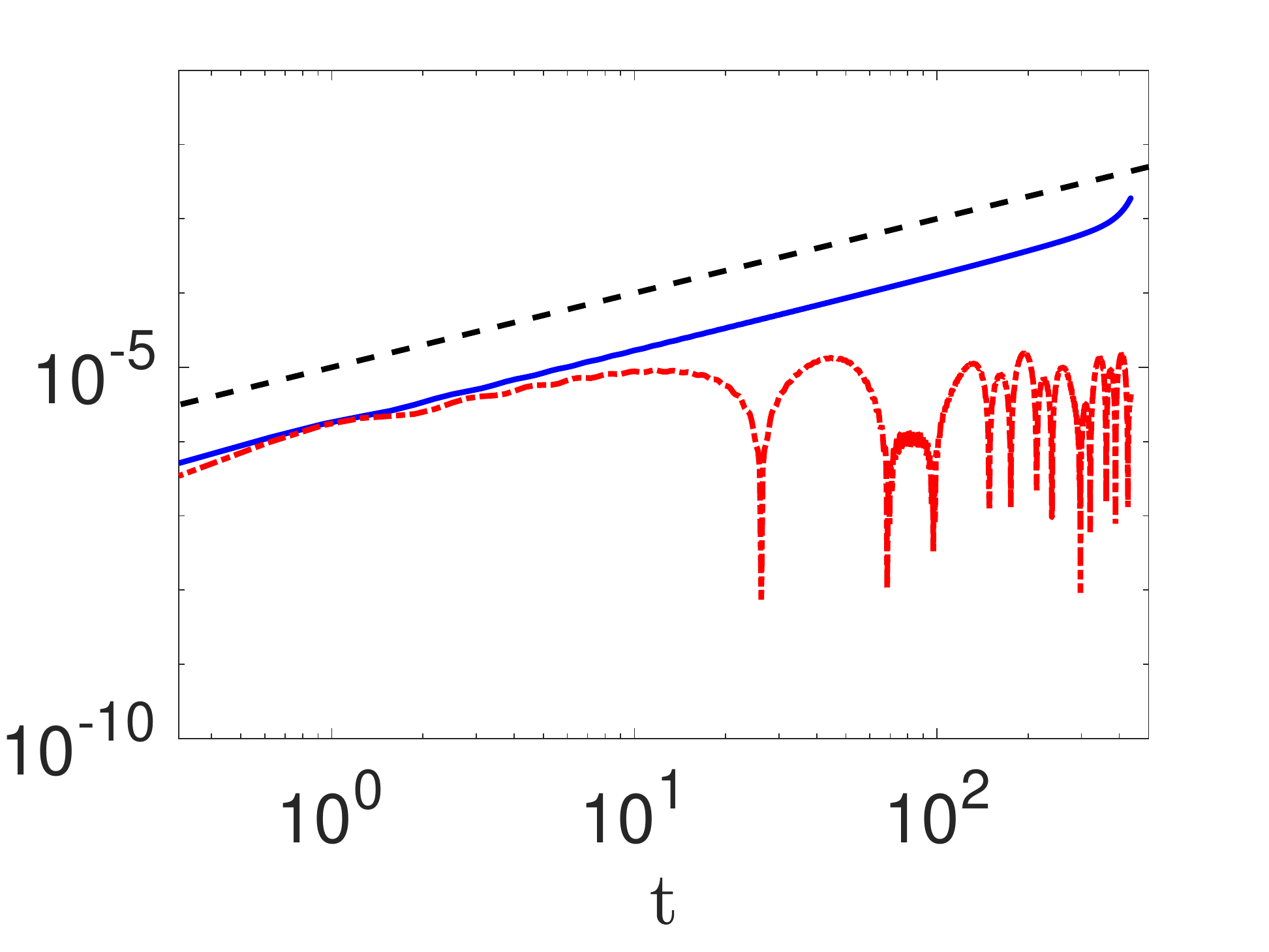}
		\subcaption{Casimir error}
		\label{fig:errors_TW_C}
	\end{subfigure}
	\begin{subfigure}{0.32\textwidth}
		\centering
		\includegraphics[width=\textwidth]{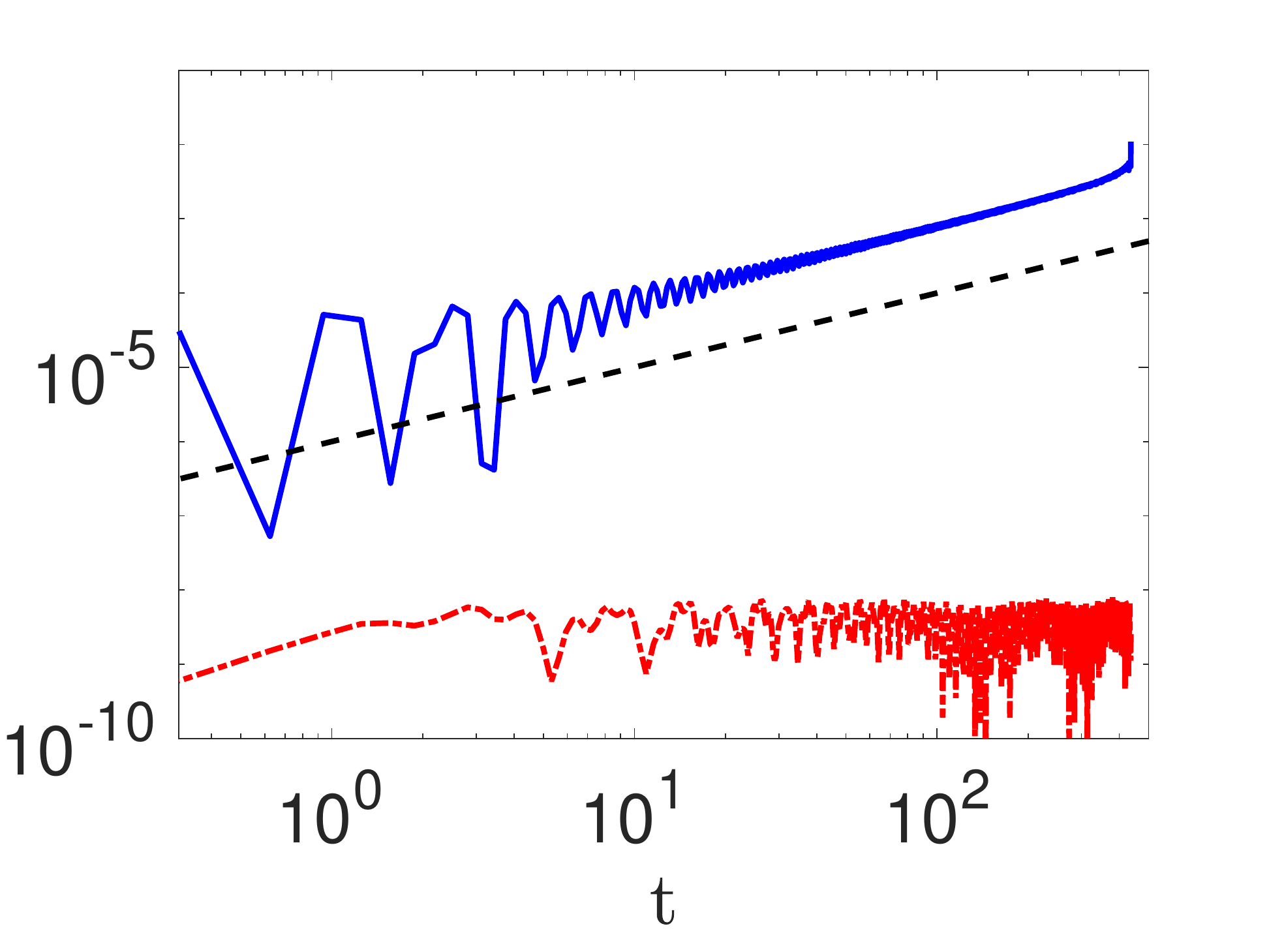}
		\subcaption{Hamiltonian error}
		\label{fig:errors_TW_H}
	\end{subfigure}	
	\begin{subfigure}{0.32\textwidth}
		\centering
		\includegraphics[width=\textwidth]{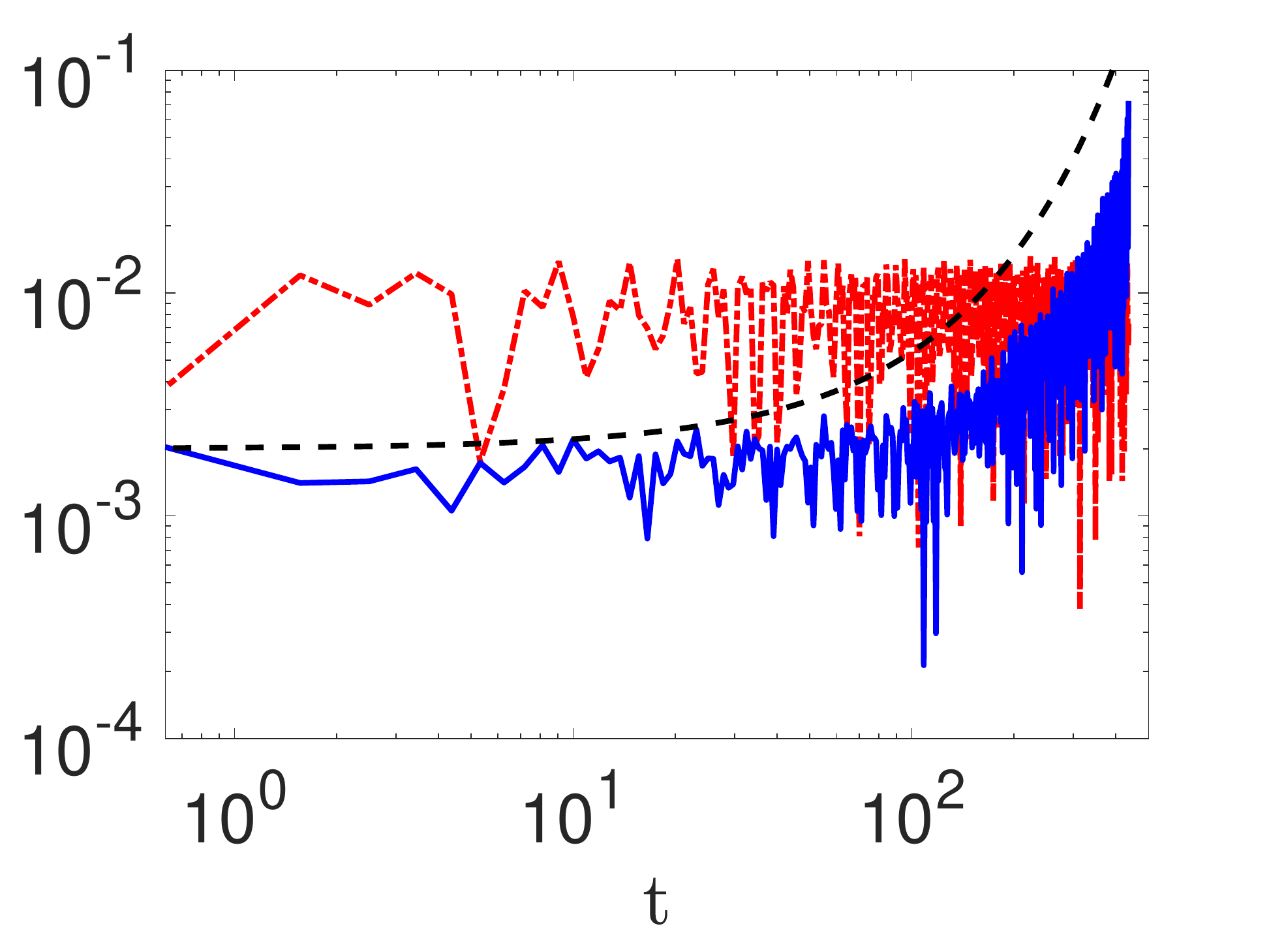}
		\caption{Highest frequency Fourier mode (increased vertical scale)}\label{fig:FM_P}
	\end{subfigure}
	\caption{The errors corresponding to the conventional (\ConvLine) and collective (\CollLine) methods for the travelling wave experiment. The reference lines (\RefLine) are $\mathcal{O}(t)$ in figures (a) and (b) and exponential in figure (c). }\label{fig:errors_TW}
\end{figure}

\subsubsection{Periodic bump solutions}
In this example, we model solutions to the extended Burgers' equation from the initial condition 
\begin{equation*}
u(x,0) = 1+ \frac{1}{2}\exp(-\sin^2(\frac{\pi x}{L})).
\end{equation*}
Figure \ref{fig:PB_P} shows snapshots of the solution and its positive Fourier modes and figure \ref{fig:errors_PB} shows the behaviour of the Casimir and Hamiltonian errors over time. Like the travelling wave example, we see that the high frequency modes of the conventional solution grow with time, which can be seen as rough wiggles in figure \ref{pb3}. The conventional solution has bounded Hamiltonian error, despite linear and exponential growth in the Casimir and highest frequency Fourier modes, respectively. In particular, the collective solution has excellent error behaviour, which appears to be bounded over the simulation period for all three plots of figure \ref{fig:errors_PB}. 
\begin{figure}
	\centering
	\begin{subfigure}{0.32\textwidth}
		\centering
		\includegraphics[width=\textwidth]{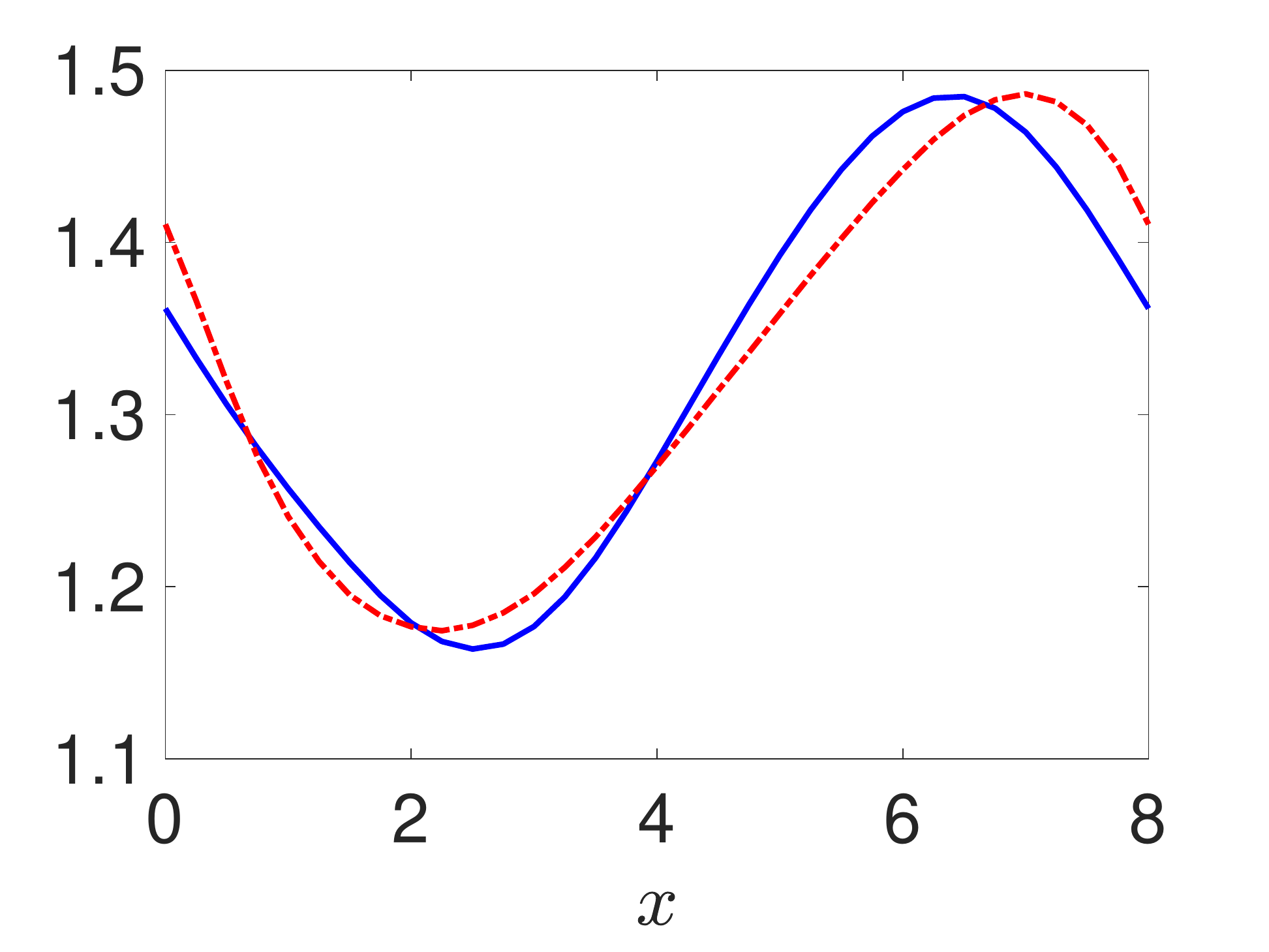}
		\subcaption{}
		\label{}
	\end{subfigure}
	\begin{subfigure}{0.32\textwidth}
		\centering
		\includegraphics[width=\textwidth]{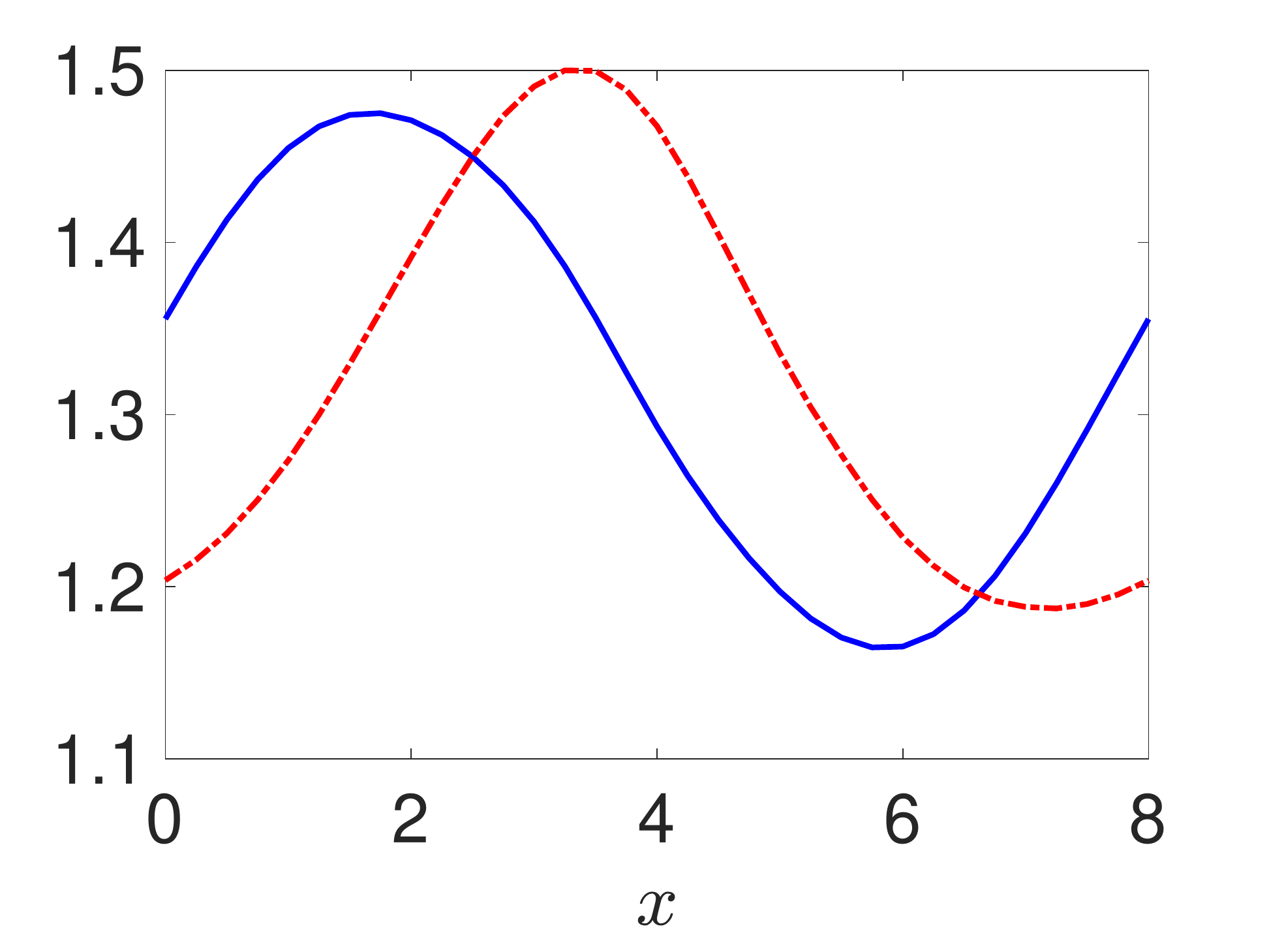}
		\subcaption{}
		\label{}
	\end{subfigure}
	\begin{subfigure}{0.32\textwidth}
		\centering
		\includegraphics[width=\textwidth]{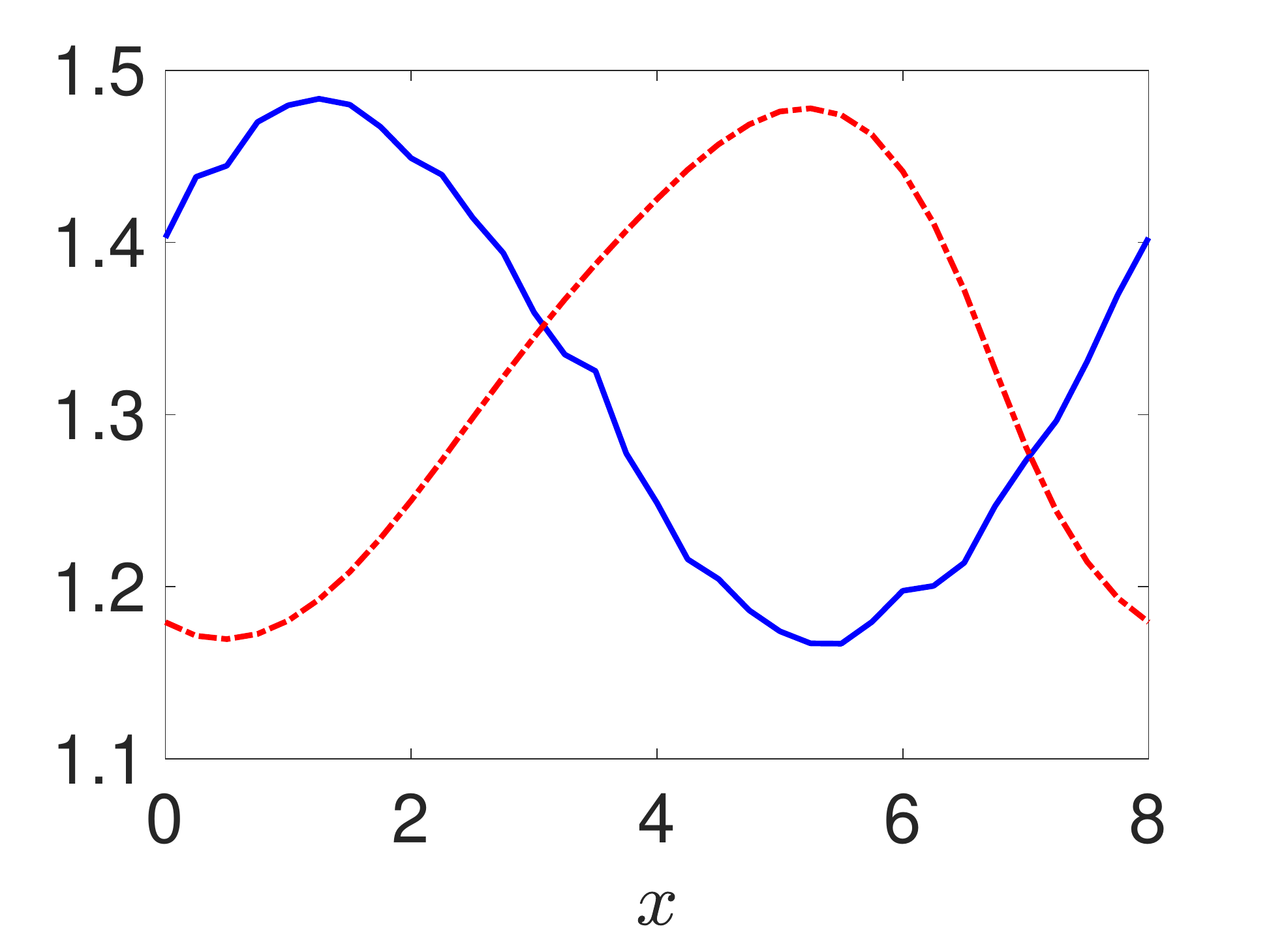}
		\subcaption{}
		\label{pb3}
	\end{subfigure}
	
	\begin{subfigure}{0.32\textwidth}
		\centering
		\includegraphics[width=\textwidth]{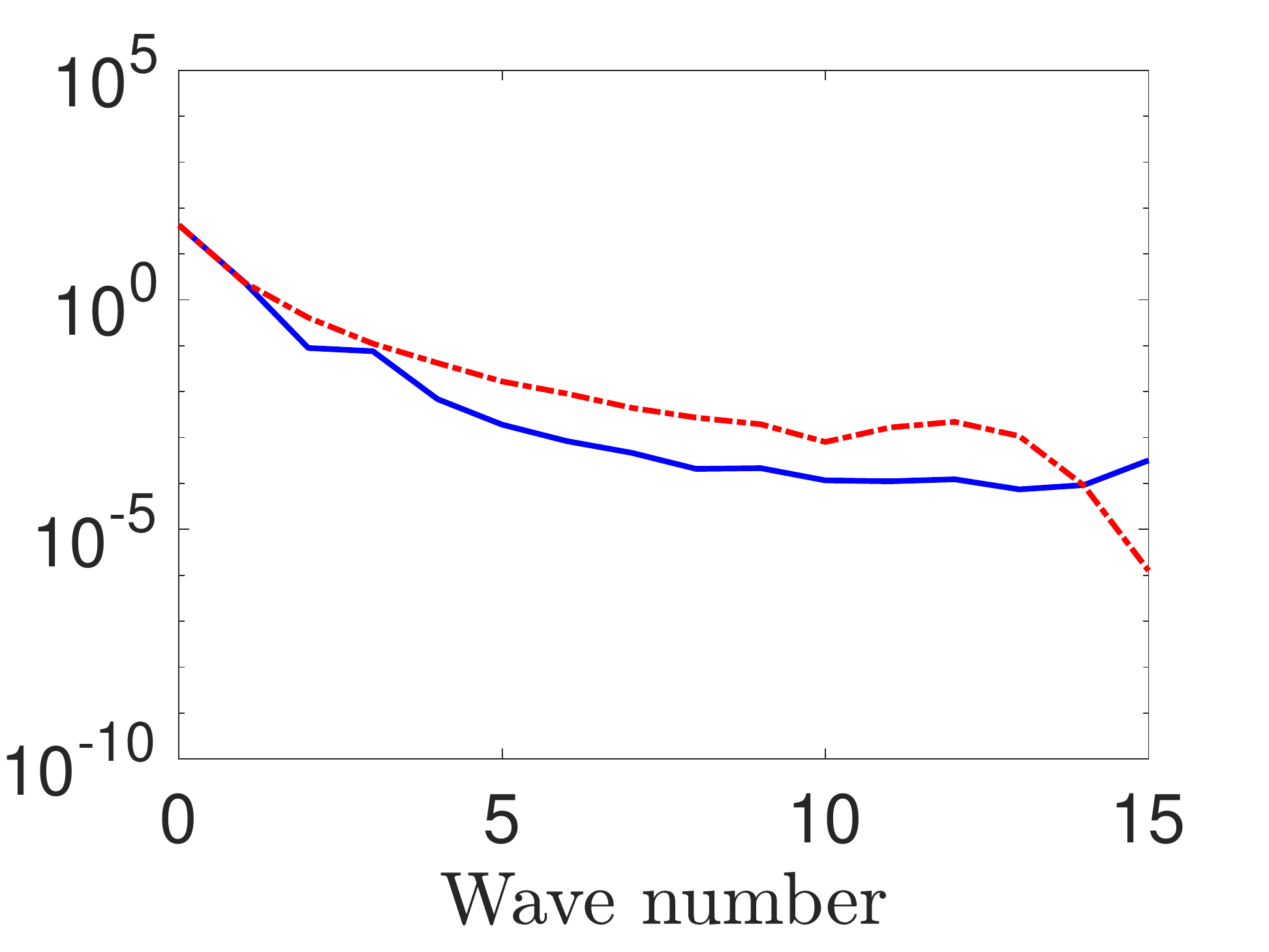}
		\subcaption{}
		\label{}
	\end{subfigure}
	\begin{subfigure}{0.32\textwidth}
		\centering
		\includegraphics[width=\textwidth]{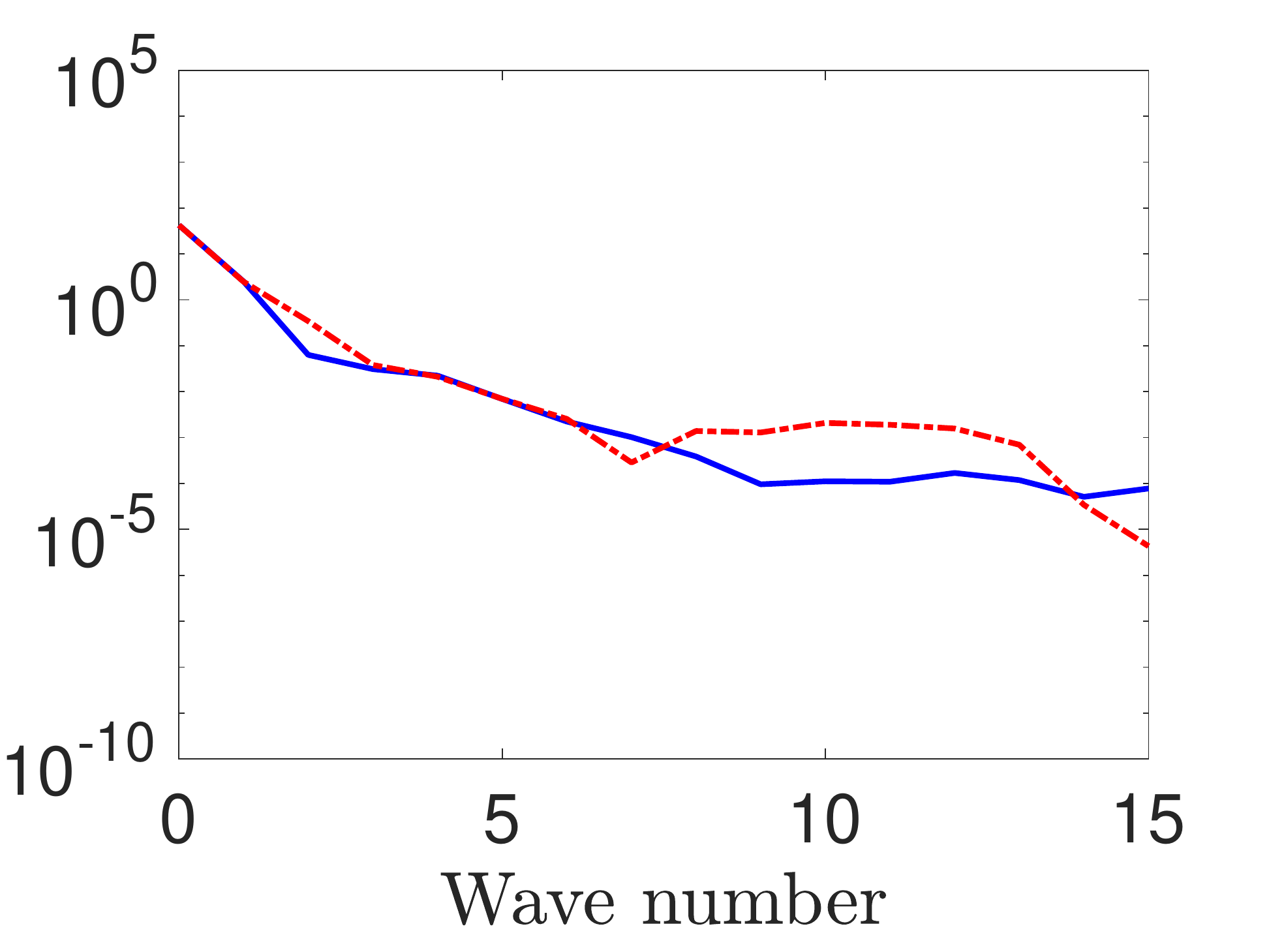}
		\subcaption{}
		\label{}
	\end{subfigure}
	\begin{subfigure}{0.32\textwidth}
		\centering
		\includegraphics[width=\textwidth]{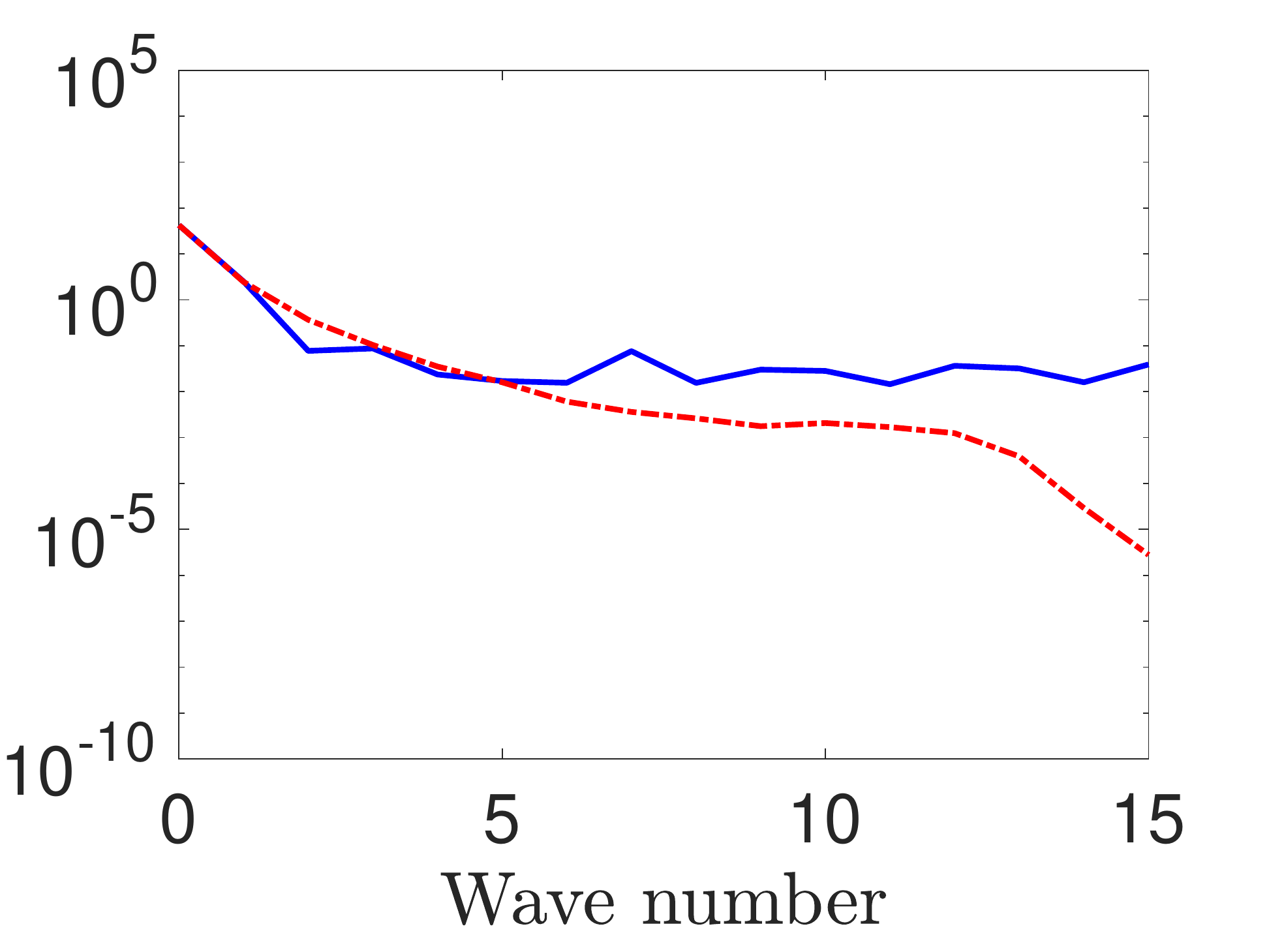}
		\subcaption{}
		\label{}
	\end{subfigure}	
	\caption{ Periodic bump solutions of the extended Burgers' equation (top row) and the positive Fourier modes (bottom row) at $t=10$ (left column), $t=100$ (middle column) and $t=1000$ (right column). The plots correspond to the conventional method (\ConvLine) and the collective method (\CollLine). The grid parameters are $n_x = 32$, $\Delta x = 0.25$, $L = 8$ and $\Delta t = 2^{-8}$.}\label{fig:PB_P}
\end{figure}
\begin{figure}
	\centering
	\begin{subfigure}{0.32\textwidth}
		\centering
		\includegraphics[width=\textwidth]{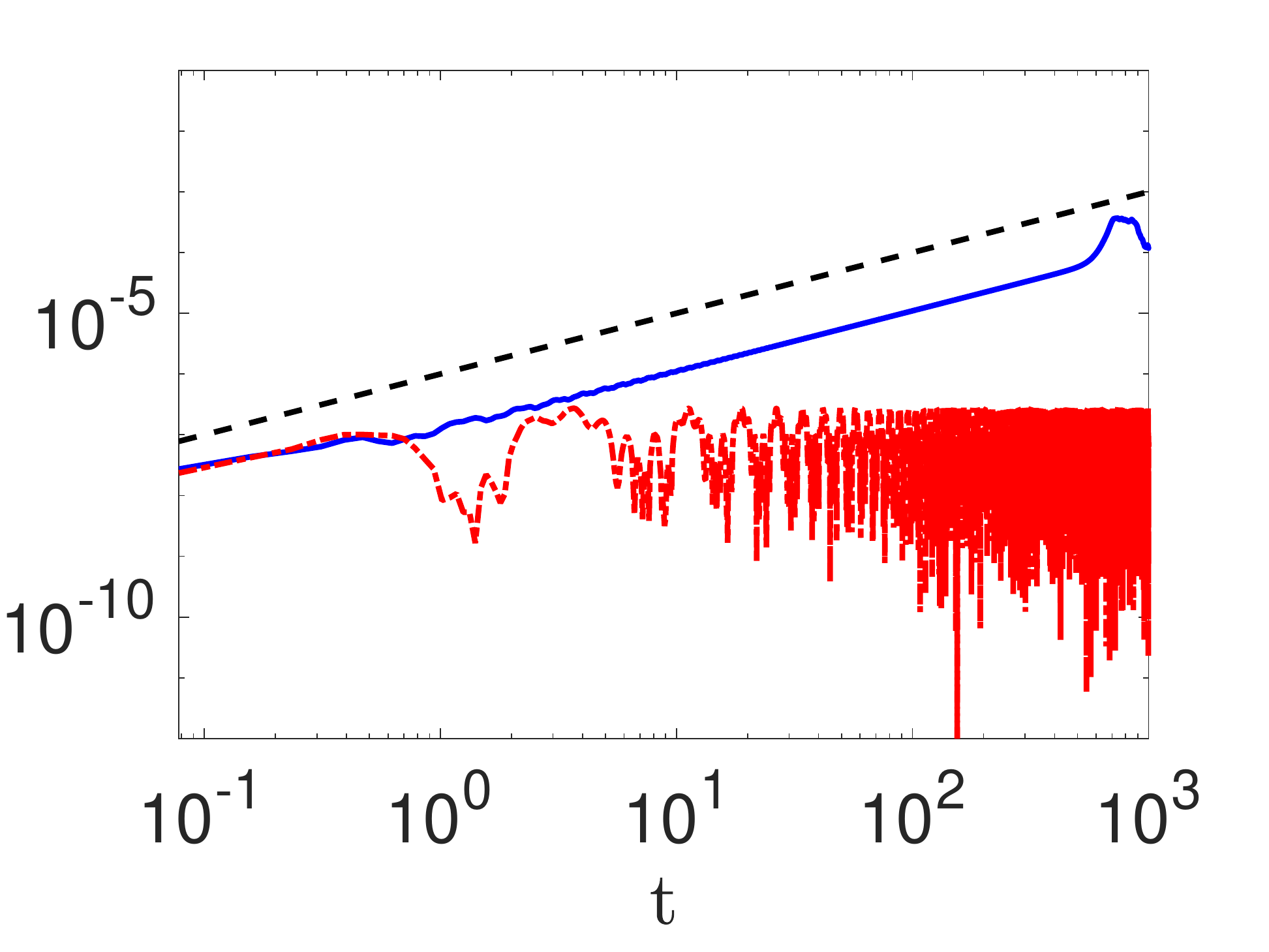}
		\subcaption{Casimir error}
		\label{}
	\end{subfigure}
	\begin{subfigure}{0.32\textwidth}
		\centering
		\includegraphics[width=\textwidth]{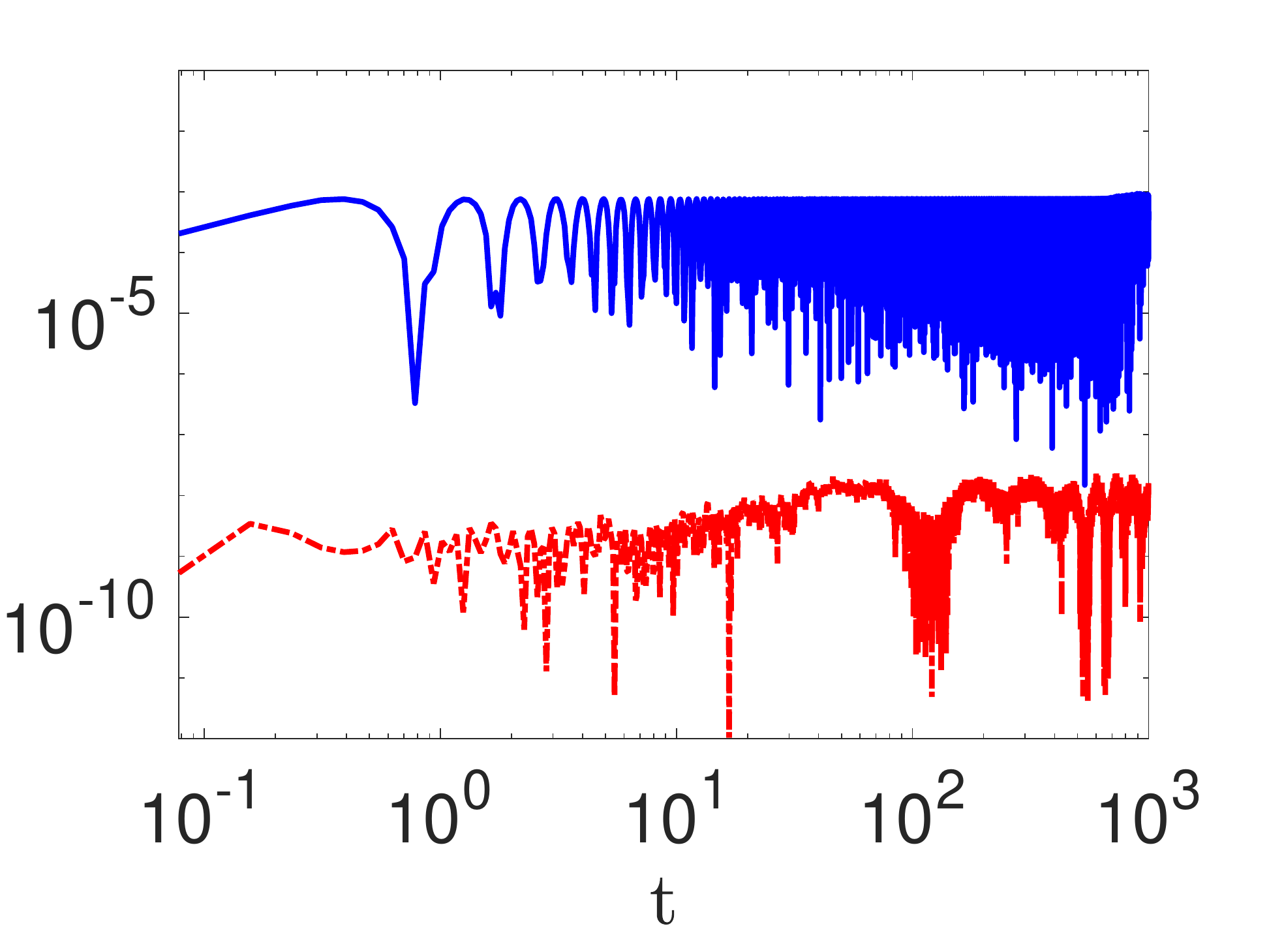}
		\subcaption{Hamiltonian error}
		\label{}
	\end{subfigure}
	\begin{subfigure}{0.32\textwidth}
		\centering
		\includegraphics[width=\textwidth]{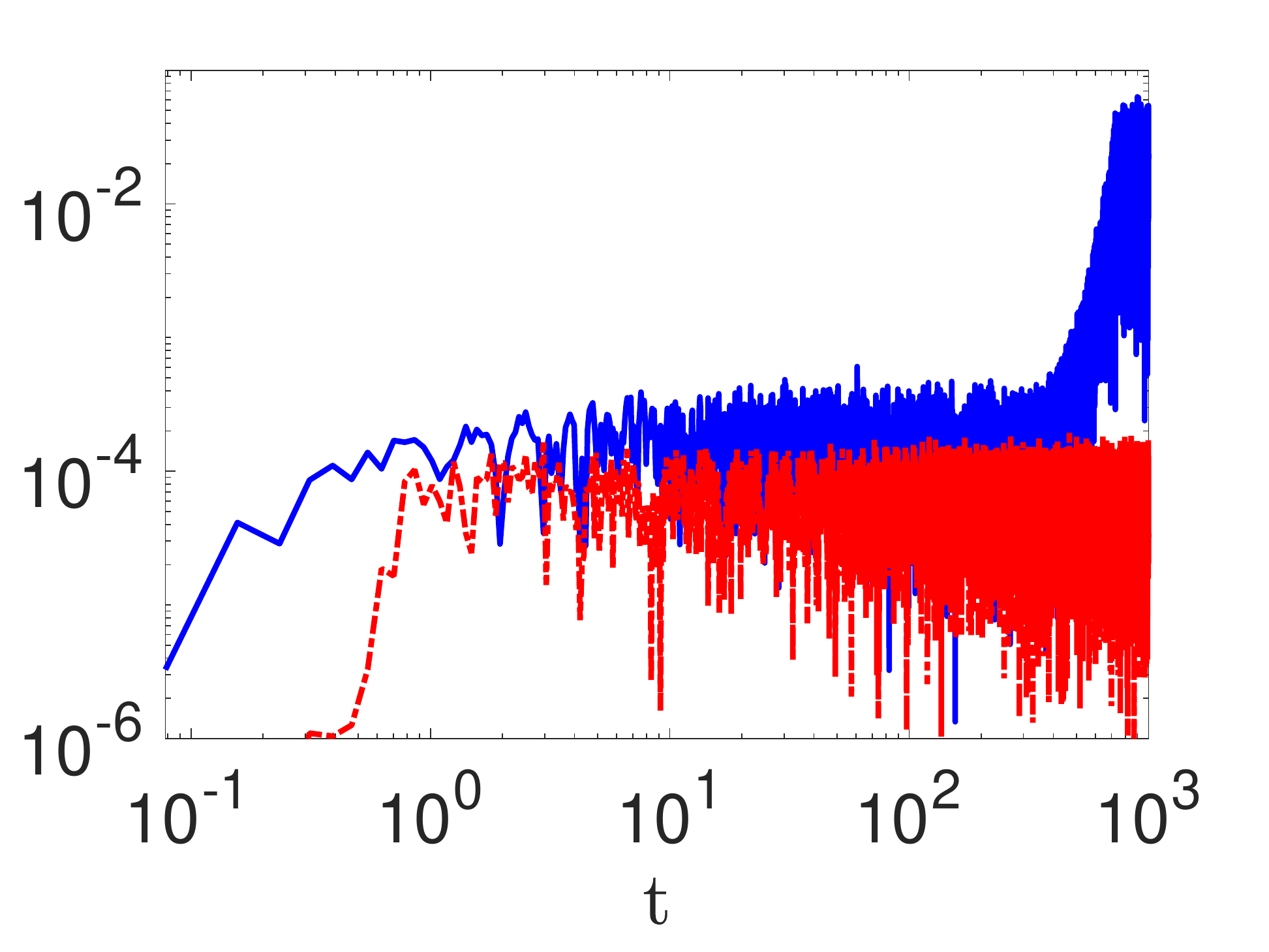}
		\caption{Highest Fourier frequency mode (increased vertical scale)}\label{fig:PB_FM_P}
	\end{subfigure}
	\caption{The errors corresponding to the conventional (\ConvLine) and collective (\CollLine) methods for the periodic bump example. The reference line (\RefLine) in figure (a) is $\mathcal{O}(t)$.}\label{fig:errors_PB}
\end{figure}

\section{Conclusion}

We have demonstrated that Hamiltonian PDEs on Poisson manifolds can be integrated while maintaining the structure preserving properties of Poisson systems very well. This is achieved by
\begin{enumerate}
\item
realising the Poisson-Hamiltonian system as an infinite-dimensional, collective Hamiltonian system on a symplectic manifold and lifting the initial condition from the Poisson system to the collective system,
\item
discretising the collective system in space to obtain a system of Hamiltonian ODEs, and
\item
using a symplectic integrator to solve the system.
\end{enumerate}

The symplectic integrator will, in general, fail to preserve the fibration provided by the realisation. Therefore, the presented integrators for Hamiltonian PDEs cannot be expected to conserve the Poisson structure {\em exactly}.
This is in contrast to the case of Hamiltonian ODEs on Poisson manifolds, where the fibres can be structurally simple for carefully chosen realisations and genuine Poisson integrators can be constructed. 
Regardless, in the ODE as well as in the PDE case the integrator is guaranteed to inherit the excellent energy behaviour from the symplectic integrator which is applied to the collective system. Moreover, our numerical examples for Hamitonian PDEs show excellent Casimir behaviour as well. Indeed, energy as well as Casimir errors are bounded in long term simulations.

Structure preserving properties of conventional numerical schemes typically rely on the presence of structurally simple symmetries of the differential equation. If the discretisation is invariant under the same symmetry as the equation, then the numerical solution will share all geometric features of the exact solution which are due to the symmetry. 
The simple form of the symmetries, however, is immediately destroyed when higher order terms in the Hamiltonian are switched on. Although exact solutions still preserve the Hamiltonian, numerical solutions obtained using a traditional scheme fail to show a good energy behaviour. The advantage of the presented integration methods is that their excellent energy behaviour is guaranteed no matter how complicated the Hamiltonian is. Our numerical examples for the extended Burgers' equation demonstrate the importance of structure preservation: while growing energy errors of the conventional solution cause a blow up, there are no signs of instabilities for the collective solution.

\section*{Acknowledgements}
We thank Elena Celledoni and Brynjulf Owren for many useful discussions. This research was supported by the Marsden Fund of the Royal Society Te Ap\={a}rangi. The third author would like to acknowledge funding from the European Unions Horizon 2020 research and innovation programme under the Marie Sklodowska-Curie grant agreement (No. 691070).


\providecommand{\href}[2]{#2}
\providecommand{\arxiv}[1]{\href{http://arxiv.org/abs/#1}{arXiv:#1}}
\providecommand{\url}[1]{\texttt{#1}}
\providecommand{\urlprefix}{URL }

\end{document}